\documentclass[12pt]{article}

\usepackage{amsthm}
\usepackage{amsmath}
\usepackage{amssymb}
\usepackage[margin=1in]{geometry}
\usepackage{enumerate}
\usepackage[colorlinks]{hyperref}
\usepackage{mathrsfs}
\usepackage{color}{\color{red}}
\usepackage{bm}
\usepackage{graphicx}

\usepackage[english]{babel}

\title{Equivalence of physical and SRB measures \\
in random dynamical systems}
\author{Alex Blumenthal\thanks{Department of Mathematics, University of Maryland, College Park, Maryland, USA. Email: alex.blumenthal@gmail.com. This material is based upon work supported by the National Science Foundation under Award No. DMS-1604805.}
\and Lai-Sang Young\thanks{Courant Institute of Math. Sciences, New York University, New York, USA.  Email: lsy@cims.nyu.edu. This research was supported in part by NSF Grant DMS-1363161.}
}

\theoremstyle{theorem}
\newtheorem{thm}{Theorem}
\newtheorem{cor}[thm]{Corollary}
\newtheorem{lem}[thm]{Lemma}
\newtheorem{prop}[thm]{Proposition}

\theoremstyle{definition}

\newtheorem{defn}[thm]{Definition}
\newtheorem{rmk}[thm]{Remark}

\newcommand{\N}{\mathbb{N}}
\renewcommand{\P}{\mathbb{P}}
\newcommand{\R}{\mathbb{R}}

\newcommand{\Z}{\mathbb{Z}}

\newcommand{\Bc}{\mathcal{B}}
\newcommand{\Fc}{\mathcal{F}}

\newcommand{\Gc}{\mathcal{G}}

\newcommand{\Sc}{\mathcal{S}}

\newcommand{\Nc}{\mathcal{N}}

\newcommand{\Qc}{{\mathcal Q}}

\renewcommand{\a}{\alpha}
\renewcommand{\b}{\beta}

\renewcommand{\d}{\delta}
\newcommand{\e}{\epsilon}

\newcommand{\pd}{\partial}

\newcommand{\diam}{\operatorname{diam}}
\newcommand{\graph}{\operatorname{graph}}

\newcommand{\Id}{\operatorname{Id}}
\newcommand{\dist}{\operatorname{dist}}

\newcommand{\Cc}{\mathcal C}
\newcommand{\Leb}{\operatorname{Leb}}
\newcommand{\Lip}{\operatorname{Lip}}
\renewcommand{\graph}{\operatorname{graph}}

\newcommand{\Pc}{\mathcal P}

\newcommand{\uo}{{\underline \omega}}

\newcommand{\calLip}{\mathcal{L}ip}
\newcommand{\Tc}{\mathcal T}

\begin{document}

\maketitle

{
{\bf Abstract. {\small We give a geometric proof, offering a new and quite different perspective
on an earlier result of Ledrappier and Young on random transformations \cite{ledrappier1988entropy}. We show that under 
mild conditions, sample measures of random diffeomorphisms are SRB measures. 
As sample measures are the limits of forward images of stationary measures, they
can be thought of as the analog of physical measures for deterministic systems.
Our results thus show the equivalence of physical and SRB measures in the random 
setting, a hoped-for scenario that is not always true for deterministic maps.}}
}

\vskip .6in
In this paper, we prove for random dynamical systems a result 
one would have liked to have for deterministic systems (referring to
systems defined by maps or flows) except that for deterministic systems,
such a result is likely not true without some additional hypotheses.

\bigskip \noindent
{\bf Ideal picture for deterministic systems}

\smallskip
To motivate our result, consider first a deterministic system on $\R^d$ (or on a finite
dimensional manifold) with an attractor. An ``ideal picture" -- which we do not
claim to be mathematically proven or even necessarily true
but which physicists often take for granted -- might be
as follows: Lebesgue measure in the basin, transported forward by the map or flow,
converges to an invariant measure on the attractor. This measure, called 
a {\it physical measure} in \cite{eckmann1985ergodic}, is the natural invariant measure from 
an observational point of view. For systems with some hyperbolicity, it is also
an {\it SRB measure},  characterized by having smooth conditional measures 
on unstable manifolds; see, e.g., \cite{eckmann1985ergodic, young2017generalizations}.

The equivalence of physical and SRB measures can be  justified 
heuristically as follows:
As mass is transported forward by a system with hyperbolicity, it is compressed 
along stable directions and spread out along unstable 
directions, eventually aligning itself with unstable manifolds. Reasoning geometrically
as we have done, it follows that the limiting 
distribution will have the SRB property. This indeed was how Ruelle
first constructed SRB measures for Axiom attractors in \cite{ruelle1974Measure}.

Reality is a little more complex outside of the Axiom A category, however: First, there is
no guarantee that the pushed forward measures will converge. Second,
Newhouse's phenomenon of infinitely many sinks \cite{newhouse1974diffeomorphisms} 
implies that for maps that are not uniformly hyperbolic,
accumulation points
of the pushed forward measures can fracture into many ergodic
components, some of which can be Dirac measures supported on sinks.
Another example to keep in mind is the figure-eight attractor
\cite{ott2008lyapunov}. This is a rather extreme example, but it points to
the fact that without adequate control, a sequence of measures that
seemingly aligns itself with unstable manifolds need not converge to
an SRB measure.

\bigskip \noindent
{\bf Random dynamical systems}

\smallskip
By a random dynamical system (RDS) in this paper, we refer to the
composition of {\it i.i.d.}  sequences of random diffeomorphisms. 
{ RDS are used to model dynamical systems with a stochastic component 
or experiencing small random fluctuations. Solutions of stochastic differential
equations (SDE) are known to have representations as stochastic flows of diffeomorphisms,
the time-$t$-maps of which are compositions of {\it i.i.d.}  sequences of random 
diffeomorphisms; see, e.g., \cite{arnold2013random, kunita}. }

In the world of RDS, it is quite natural for the stationary measure to 
have a density, so let us for the moment confuse the stationary measure 
with Lebesgue measure. Also, ergodicity is achieved easily in such RDS, and with
ergodicity, one does not have to be concerned with the fracturing of 
the limit measure. Under these assumptions, all of
which are quite mild for RDS, we prove that the reasoning in the
``ideal picture" above is valid. 

\bigskip \noindent
{\bf Main Result (informal version).} {\it Consider an ergodic RDS
$\{f^n_\uo\}$, the stationary measure $\mu$ of which has a density.
Assume the system has a positive Lyapunov exponent.
Then for almost every sample path $\uo$, $(f^n_{\theta^{-n}\uo})_*\mu$
converges as $n \to \infty $ to a random SRB measure $\mu_\uo$.
Here $\theta^n$ is time shift
on the sequence of random maps. }

\bigskip
A precise formulation is given in Section 1. Under the conditions above, 
we have also an {\it entropy equality}, which asserts that pathwise entropy 
is equal to the sum of positive Lyapunov exponents. That follows easily once 
we have the SRB property,
by a proof identical to that for deterministic systems.

The results above are not new. They were first proved by Ledrappier and Young 
\cite{ledrappier1988entropy} and subsequently extended to random endomorphisms by Liu, Qian and Zhang \cite{liu2002entropy}; see
also the more recent book \cite{qian2009smooth} of Qian, Xie and Zhu. 
In these earlier proofs, the authors showed that the RDS satisfies 
an entropy equality, from which they deduced the SRB property of $\mu_\uo$
by appealing to another theorem. This last result, which provides
the crucial link to random SRB measures, is not elementary,
especially when zero Lyapunov exponents are present;
see \cite{ledrappier1985metric} for a complete proof in the nonrandom case.
We mention also the recent result \cite{brown2017measure} of Brown and Rodriguez-Hertz 
for random surface diffeomorphisms, proved under an assumption of randomness for $E^s$.

The proof presented here is new and different, and we think it has the following merits: 
One, it is conceptually more transparent and confirms the intuition behind
the ``ideal picture" discussed above. Two, it highlights clearly the differences between deterministic
and random dynamical systems; and three, our proof is more generalizable as we will show
in forthcoming papers. For example, the proof of the entropy formula in
\cite{ledrappier1988entropy}  involves conditional densities 
on the stable foliation,
ruling out immediately direct generalizations to semiflows defined
by dissipative PDEs, for which stable manifolds are always infinite
dimensional.

Finally, one of our motivations for presenting a more accessible proof
is that there has been some renewed interest in random dynamical
systems, and in the idea of random SRB measures in particular. We mention
two recent applications in which these ideas have appeared: one is
the reliability of biological and engineered systems (see, e.g., \cite{lin2009reliability})
and the other is in climate science (see, e.g., \cite{chekroun2011stochastic}).


\section{Setting and statement of results}

We begin with the definition of a {\it random dynamical system}, abbreviated
as RDS.

Let $\Omega$ be a Polish space, and let $P$ be a Borel probability measure 
on $\Omega$. Let $M$ be a compact Riemannian manifold, and consider
 a Borel measurable mapping $\omega \mapsto f_{\omega}$ 
 from $\Omega \to$ Diff$^2(M)$, the space of $C^2$ 
diffeomorphisms from $M$ onto itself equipped with the $C^2$-metric.
An RDS consists of compositions of sequences of maps from
$\{f_{\omega}, \omega \in \Omega\}$
chosen {\it i.i.d.} with law $P$. 
For $\uo = (\omega_n)_{n \in \Z} \in \Omega^\Z$ and $n \in \Z$, we write
\[
f_{\uo}^n = 
\begin{cases}
f_{\omega_n} \circ \cdots \circ f_{\omega_1} & n > 0 \\
\Id & n = 0 \\
f_{\omega_{-(n-1)}}^{-1} \circ \cdots \circ f_{\omega_0}^{-1} & n < 0
\end{cases}
\]
One considers also one-sided compositions $f_{\uo^+}^n$
for $\uo^+ \in \Omega^{\Z^+} := \prod_{n > 0 } \Omega$ and $n >0$.

There are several ways to view an RDS. One is as a Markov chain $(X_n)$ 
on $M$ defined by fixing an initial condition $X_0 \in M$ and setting 
$X_{n+1} = f_{\omega_{n+1}}(X_n)$. 
Equivalently, we define the transition probabilities of the chain by
$$
\Pc(E|x) = P\{\omega: f_\omega x \in E\}
$$
for $x \in M$ and Borel sets $E \subset M$.
A Borel probability measure $\mu$ on $M$ is said to be {\it stationary} 
if for all Borel sets $E \subset M$, 
$$
\mu(E) = \int \Pc(E|x) \mu (dx)\ .
$$

Another viewpoint is to represent an RDS as a measure-preserving skew product
map. Here it is important to distinguish between the two-sided and one-sided
cases. Let $\theta: \Omega^\Z \to \Omega^\Z$ be the leftward shift
preserving the probability $\mathbb P = P^\Z$ on $\Omega^\Z$, and let 
$\theta^+: \Omega^{\Z^+} \to \Omega^{\Z^+}$ be the corresponding shift 
preserving $\mathbb P^+ = P^{\Z^+}$. Then the skew product maps 
corresponding to the RDS above are given by
\begin{eqnarray*}
 \tau : M \times \Omega^\Z \to M \times \Omega^\Z
\quad & \mbox{with}& \quad \tau (x, \uo) = (f_{\omega_1} x, \theta \uo)\\
\mbox{and} \qquad \tau^+ : M \times \Omega^{\Z^+} \to M \times \Omega^{\Z^+}
\quad &\mbox{with}& \quad \tau^+(x, \uo^+) = 
(f_{\omega_1} x, \theta^+ \uo^+)\ ,
\end{eqnarray*}
and Lemma \ref{lem:invariantMeasures1} identifies the relevant invariant measures of $\tau$ and $\tau^+$:

\newpage
\begin{lem}\label{lem:invariantMeasures1} \
\begin{itemize}
\item[(a)] A Borel probability measure $\mu$ on $M$ is a stationary measure of the 
Markov chain $(X_n)$ if and only if $\mu \times \P^+$ is an invariant
measure of $\tau^+ : M \times \Omega^{\Z^+} \to M \times \Omega^{\Z^+}$.
\item[(b)] Given $\mu$ as above, there is a unique $\tau$-invariant probability measure $\mu^*$ on $M \times \Omega^\Z$ that projects onto $\mu \times \P^+$.
\end{itemize}
\end{lem}

The next lemma gives more information on the disintegration 
of $\mu^*$ on $M$-fibers, i.e., the family of probability measures
$\{\mu_\uo, \uo \in \Omega^\Z\}$ on $M$ with the property that for all continuous
$\varphi: M \times \Omega^\Z \to \mathbb R$, we have
$$
\int \varphi (x, \uo) d \mu^*(x,\uo)  = \int \left(\int \varphi (x, \uo) d\mu_\uo (x)\right)\ 
d \mathbb P (\uo)\ .
$$

\begin{lem} \label{lem:invariantMeasures} \
\begin{itemize}
\item[(a)] The measures $\mu_\uo$ are invariant in the sense that for each 
$\uo = (\omega_n) \in \Omega^\Z$, 
$$
(f_{\omega_1})_* \mu_\uo = \mu_{\theta \uo}\ .
$$
\item[(b)] For $\P$-a.e. $\uo \in \Omega^\Z$,
\[
(f^n_{\theta^{-n} \uo})_* \mu \to \mu_{\uo} \qquad \mbox{weakly as } n \to \infty\ .
\]
It follows that $\mu_{\uo}$ depends only on $\omega_n$ for $n \le 0$.
\end{itemize}
\end{lem}
Lemmas \ref{lem:invariantMeasures1} and \ref{lem:invariantMeasures} are standard; see, e.g., Chapter 1 of \cite{arnold2013random} for details.

Lemma \ref{lem:invariantMeasures}(b) tells us that the $\mu_\uo$, which are called  {\it sample measures}, are 
in fact  the conditional distributions of $\mu$ given the history of the dynamical system, 
$\uo^- = (\omega_n)_{n \le 0}$. Intuitively, they 
represent what we see at time $0$ given that the transformations $f_{\omega_n}, 
n \le 0$, have occurred. 

Given an RDS together with a stationary measure $\mu$, certain properties of
deterministic systems $(f, m)$, where $f$ is a single diffeomorphism and $m$
an invariant measure, extend in a straightforward way to the RDS
via their skew product representations. 
We assume throughout that 
\begin{align}\label{eq:integrability}
\int \log^+ \|  f_\omega \|_{C^2} \, d P(\omega) \, , \quad \int \log^+ \| f^{-1}_\omega\|_{C^2} \, d P(\omega) \ \ < \infty \, .
\end{align}
These conditions are satisfied by the time-one maps of a large class of SDEs \cite{kifer1988}. Under these assumptions, 
the following are known: For one-sided skew products,
Lyapunov exponents of $f^n_{\uo^+}$ are defined $\mu$-a.e. for $\mathbb P^+$-a.e.
$\uo^+$, as are stable manifolds corresponding to
negative Lyapunov exponents. For the two-sided skew-product, Lyapunov
exponents of $f^n_\uo$ are defined $\mu^*$-a.e., as are stable 
and unstable manifolds. Lyapunov exponents are nonrandom. Another nonrandom
quantity of the RDS is pathwise entropy, which we denote by $h_\mu(\{f^n_{\uo^+}\})$. 
See \cite{arnold2013random, kifer2012ergodic} for more information.

As a direct generalization of the idea of SRB measures in the deterministic case, we have
the following:

\begin{defn}\label{defn:SRB} Let $\{f_\omega\}$ and $\mu$ be given.
We say the $\mu_\uo$ are {\rm \bf random SRB measures} if \

1. $f^n_\uo$ has a positive Lyapunov exponent $\mu^*$-a.e. and 

2. for $\mathbb P$-a.e. $\uo$, the sample measure $\mu_\uo$ has absolutely continuous conditional
measures on unstable manifolds.
\end{defn} 

The main result of this paper can now be stated formally as follows:

\bigskip \noindent
{\bf Main Theorem.} {\it Let $\{f_\omega\}$ be a RDS satisfying \eqref{eq:integrability}, and let $\mu$ be an ergodic
stationary measure. We assume that

1. $\mu \ll \Leb$ with a continuous density, and 

2. $\{f^n_{\uo^+}\}$ has a positive Lyapunov exponent $(\mu \times \mathbb P^+)$-a.e. 

\noindent Then the $\mu_\uo$ are random SRB measures. }

\medskip

\bigskip \noindent
{\bf Corollary.} {\it 
Let $\{ f_\omega\}$ be as in the Main Theorem. Then the entropy formula  
\[
h_\mu(\{f^n_{\uo^+}\}) = \sum_{i : \lambda_i > 0} m_i \lambda_i \, 
\]
holds. Here $h_\mu(\{ f^n_{\uo^+}\})$ is pathwise entropy, and $\lambda_1 > \lambda_2 > \cdots > \lambda_d$ denote the Lyapunov exponents of $(f^n_{\uo^+})$ with multiplicities $m_i, 1 \leq i \leq d$.}

\bigskip
As noted in the Introduction, the results above were first proved 
 in \cite{ledrappier1988entropy}. They were subsequently extended 
 to random endomorphisms in \cite{liu2002entropy}, and to  compositions
  that are not necessarily {\it i.i.d.} in \cite{qian2009smooth}. 
In all of these papers, the result in the Corollary is first proved, 
and the result in the Main Theorem is deduced from that by appealing to
the RDS version of the entropy formula characterization for SRB measures.
Here we prove these results in the opposite order: we give a direct proof
of the SRB property of $\mu_\uo$. Once that is proved, the Corollory  follows
immediately by a proof identical to that in the deterministic case.

Our proof of the Main Theorem will proceed as follows. For $\mathbb P$-a.e. $\uo$,
we consider $(f^n_{\theta^{-n}\uo})_*\mu:= \mu^n_\uo$, which we know converges to
$\mu_\uo$ as $n \to \infty$ by Lemma \ref{lem:invariantMeasures}(b).
It suffices to show that $\mu_\uo$ has smooth conditional probabilities on unstable 
manifolds, and we will prove that by showing that the geometric argument in the 
``ideal picture" in Section 1 can, in fact, be made rigorous for RDS. 

One of the technical novelties of this paper is our analysis of orbits with {\it finite pasts}.
For RDS, this is both important and natural, for the set of ``typical" points changes with knowledge of
the past: with zero knowledge of the past, $\mu$-a.e. $x$ is ``typical"; 
starting from time $-n$, 
typicality as seen at time $0$ is with respect to $\mu^n_\uo$, 
and as $n \to \infty$, this measure becomes $\mu_\uo$.


\bigskip

 \noindent
{\it The following notation will be used throughout:}
\begin{itemize}
\item[--] On $M$: $T_xM$ is the tangent space at $x$, $\|\cdot\|$ is the norm on $T_xM$,
$d(\cdot, \cdot)$ is the distance on $M$ inherited from the
Riemannian metric, and $\mathcal B(x,r) = \{y \in M: d(x,y) <r\}$.
\item[--] If $E \subset T_x M$ is a subspace, then $E(r) = \{v \in T_x : \|v\| \le r\}$.
\item[--] On $\R^d, d \geq 1$, norms are denoted ${| \cdot |}$, and balls centered 
at the origin by $B(\cdot)$; see 
Sect. 2.1 for detail.
\end{itemize}

\section{Preliminaries and Main Proposition} 

In this section, we consider exclusively the two-sided skew product 
$$\tau : M \times \Omega^\Z \to M \times \Omega^\Z \qquad
\mbox{given by} \qquad
\tau (x, \uo) = (f_{\omega_1} x, \theta \uo) 
$$
with invariant probability measure $\mu^*$. {Sects. 2.1--2.4 contain
some preliminary facts that will be used later on. In Sect. 2.5,} we formulate
the Main Proposition (Proposition \ref{prop:main}) and explain why 
it implies the Main Theorem. The proof of Proposition \ref{prop:main} will
occupy the rest of this paper.

\subsection{Two-sided charts for random maps (mostly review)}

Assuming the existence of a strictly positive Lyapunov exponent, we first
record some properties enjoyed by two-sided Lyapunov charts 
at $\mu^*$-a.e. $(x,\uo)$ for the skew-product map $\tau$. 
Details of chart construction will be omitted as the results are entirely analogous 
to those for deterministic maps, and such charts have been used before for RDS
(see, e.g., \cite{arnold2013random}, Chapter 4 for more detail). 
We will include only those properties that are relevant for subsequent discussion.

\begin{prop} [Linear picture] There exist $\lambda_0>0$ and a $\tau$-invariant Borel measurable subset 
$\Gamma \subset M \times \Omega^\Z$ with $\mu^*(\Gamma)=1$ 
such that on $\Gamma$ there is a measurable splitting
$$
(x,\uo) \mapsto E^u_{(x, \uo)} \oplus E^{cs}_{(x, \uo)} = T_xM
$$
with respect to which the following hold for each $(x, \uo) \,  : $
\begin{itemize}
\item[(a)] $\lim_{n \to \infty} { \frac1n \log \| df^{-n}_{(x, \uo)} |_{E^u_{(x, \uo)}}\|} = -\lambda_0 \,  \, ;$
\item[(b)] $\lim_{n \to \infty} \frac1n \log \| df^n_{(x, \uo)} |_{E^{cs}_{(x, \uo)}} \| \leq 0 \, \, ; $  and
\item[(c)] $\lim_{n \to \pm \infty} \frac{1}{|n|} \log \| \pi^{u/cs}_{\tau^n (x, \uo)} \| = 0 \, $. 
\end{itemize}
Here, $\pi^{u/cs}_{(x, \uo)}$ denotes the projection onto $E^{u/cs}_{(x, \uo)}$ 
along $E^{cs/u}_{(x, \uo)} \, .$
\end{prop}

\medskip

\newcommand{\wLip}{\widetilde{\operatorname{Lip}}}

Below we formulate a system of adapted charts for the two-sided dynamics. 
Let $\R^u = \R^{\dim E^u}, \R^{cs} = \R^{\dim E^{cs}}$ (recall that 
since $\mu$ is an ergodic stationary measure, hence $(\tau, \mu^*)$ is ergodic, 
we have that $\dim E^{u/cs}_{(x, \uo)}$ is constant along $\Gamma$).
For $w = u + v \in \R^u \times \R^{cs}$, we define $|w| = \max \{ | u |, | v | \}$ where 
$|u|$ and $|v|$ are Euclidean norms on $\R^u$ and $\R^{cs}$ respectively.
For $r > 0$, we let $B^{u/cs}(r)=\{v \in \R^{u/cs}: |v| \le r\}$, and write $B(r) = B^u(r) + B^{cs}(r)$. 

\begin{prop}[Nonlinear picture] \label{prop:charts2Side}
Fix $\d_0, \d_1, \d_2 > 0$ with $\d_0, \d_2 \ll \lambda_0$ and $\d_1$ sufficiently small, and let $\lambda = \lambda_0 - \d_0$. Shrinking $\Gamma$ by a set of $\mu^*$-measure 0
(and continuing to call it $\Gamma$), there are defined on $\Gamma$
\begin{itemize}
\item[(i)] a Borel measurable family of invertible linear maps 
\[
L_{(x, \uo)} : \R^u \times \R^{cs} \to T_x M \, , 
\]
with $L_{(x, \uo)}\R^u = E^u_{(x, \uo)}$ and $L_{(x, \uo)}\R^{cs} = E^{cs}_{(x, \uo)}$, and
\item[(ii)] a measurable function
$l: \Gamma \to \mathbb [1,\infty)$ satisfying $e^{- \d_2} \leq \frac{l \circ \tau}{l} \leq e^{\d_2}$  , 
\end{itemize}
with respect to which the following hold. Let the {\rm chart} at  $(x, \uo)$ be given by
$$
\Phi_{(x,\uo)}:  B(\d_1 l(x, \uo)^{-1}) \to M \qquad \mbox{with} \qquad 
\Phi_{(x, \uo)} = \exp_x \circ L_{(x, \uo)}\ ,
$$
and define the {\rm connecting maps} between charts to be 
$$
\tilde f_{(x, \uo)} = \Phi_{\tau(x, \uo)}^{-1} \circ f \circ \Phi_{(x, \uo)}:  B(\d_1 l(x, \uo)^{-1}) \to \R^u \times \R^{cs}\ .
$$
Then (a) for any $y, y' \in \Phi_{(x, \uo)}  B(\d_1 l(x, \uo)^{-1})$, we have
\begin{gather*}
d(y, y') \leq |\Phi_{(x, \uo)}^{-1} y - \Phi_{(x, \uo)}^{-1} y' | \leq l(x, \uo) d(y, y') \, ;
\end{gather*}
and (b) $\tilde f_{(x, \uo)}$ satisfies
\begin{itemize}
\item[(b1)]
$| (d \tilde f_{(x, \uo)})_0 u| \geq e^{\lambda} |u|$ for $u \in \R^u$ \ \ \  
and \ \ \ $| (d \tilde f_{(x, \uo)})_0 v| \leq e^{\d_0} |v|$ for $v \in \R^{cs}$;
\item[(b2)] 
$\Lip( \tilde f_{(x, \uo)} - (d \tilde f_{(x, \uo)})_0 |_{ B(\d l(x, \uo)^{-1})}) \leq \d$ 
\quad for all $\d \in (0,\d_1)$\ ; and 
\item[(b3)] $\Lip (d \tilde f_{(x, \uo)}) \leq l(x, \uo)$.
\end{itemize}
\end{prop}

A difference in Proposition \ref{prop:charts2Side} from the single diffeomorphism case
is that in (b2) and (b3) above, we needed to take into consideration the possibly unbounded
sequence of $C^2$ norms $\| f_{\omega_n}^\pm\|_{C^2}, n \in \Z$. We account for this
 by taking $l \geq l_1$, where \[l_1(\uo) := \sup_{n \in \Z} \big( e^{- |n| \d_2} \max \{ \| f_{\omega_n}\|_{C^2}, \| f_{\omega_n}^{-1} \|_{C^2}\}\big) \]
is finite $\P$-almost surely by our integrability condition \eqref{eq:integrability}.

As in the deterministic case, we also have the notion of {\it uniformity sets},
i.e., sets of
the form $\Gamma_{l_0}:= \{l \le l_0\}$ for fixed $l_0 \geq 1$. 

\subsection{Continuity of $E^u$ and $E^{cs}$}

For a single diffeomorphism, the continuity of $E^u$ and $E^{cs}$ on uniformity sets is well known.
We formulate and prove here the RDS versions that will be needed later on. 
For $\uo \in \Omega^\Z$, 
we write $\uo = (\uo^-, \uo^+) \in \Omega^{\Z^-} \times \Omega^{\Z^+}$,
where $\Omega^{\Z^-} = \prod_{n \le 0} \Omega$ and 
$\Omega^{\Z^+} = \prod_{n > 0} \Omega$.

\begin{prop}[Continuity of $E^u$ and $E^{cs}$]\label{prop:ctyEucs} Let $l_0 \geq 1$ be fixed. Then 
\begin{itemize}
\item[(a)] for fixed $\check \uo^+ \in \Omega^{\Z^+}$, 
$(x, \uo) \mapsto E^{cs}_{(x,\uo)}$ 
is continuous among $(x, \uo) \in \{l \le l_0 \} \cap \{ \uo^+ = \check \uo^+\}$.
\item[(b)] for fixed $\check \uo^- \in \Omega^{\Z^-}$, $(x, \uo) \mapsto E^u_{(x,\uo)}$ is continuous
among $(x, \uo) \in \{l \le l_0 \} \cap \{ \uo^- = \check \uo^-\}$. Specifically, for
any $\e > 0$, there exists $n_0 = n_0(\e, l_0)$ and $\eta = \eta(\e, l_0)$ such that if
 $(x, \uo), (y, \uo') \in \{ l \leq l_0\}$ are such that $\omega_{-i} = \omega_{-i}'$ for all
$0 \leq i \leq n_0 - 1$, then
\[
d(x,y) < \eta \quad \text{ implies } \quad d_H(E^u_{(x, \uo)}, E^u_{(y, \uo')} ) < \e \, .
\]
\end{itemize}
\end{prop}

\smallskip
Here, for $E \subset T_x M, E' \subset T_yM$, we have written 
$d_H(E, E')$ for the Hausdorff distance 
between the unit balls of $E$ and $E'$.
Since all considerations are local, we will assume, via the use of charts, that 
we are working in Euclidean space where
there is a canonical identification of tangent spaces. Part (a) is standard: $E^{cs}$ depends only on the future $\uo^+ = (\omega_i)_{i \geq 1}$.
Later on we will need the ``finite past" version of Part (b), which
says that the dependence of $E^u$ on the \emph{far past}, 
i.e., on $(\omega_{-i})_{i \geq n}$ for large $n$, is weak, 
 and we give a proof of it.

\begin{proof}[Proof of (b)] Let $x, y \in M$ be nearby points and $\uo, \uo' \in \Omega^\Z$ be such that
 $(x, \uo), (y, \uo') \in \{ l \leq l_0\}$. Assume that $\omega_{-i} = \omega_{-i}'$ for all $0 \leq i \leq n_0-1$ for some $n_0 \in \N$ to be specified. 
 Crucially, in the argument below we work exclusively with the maps $f^{-1}_{\omega_{- i}}$, $0 \leq i \leq n_0-1$. For ease of notation, let us write $f^{-i} := f_{\omega_{-(i-1)}}^{-1} \circ \cdots \circ f_{\omega_0}^{-1}$, $x_{-i} = f^{-i} x, y_{-i} = f^{-i} y$, and { $df^{-1}_{x_{-i}} = (df_{\omega_{-i}}^{-1})_{x_{-i}}$. }

 Let $v \in E^u_{(x, \uo)}$ be a unit vector and write $v = \hat v + v^s$ according to the splitting $E^u_{(y, \uo')} \oplus E^{cs}_{(y, \uo')}$ of $T_y M$. It suffices to bound $\| v^s\| \leq \e$ when $d(x, y)$ is suitably small. To begin, we estimate:
\begin{align*}
\| df^{-{n_0}}_{y} v \| & \leq \| df^{-{n_0}}_{x} v \| + \| df^{-{n_0}}_{x} - df^{-{n_0}}_{y} \|  \, .
\end{align*}
The first term is bounded $\leq l_0 e^{- n_0 (\lambda - \d_2)}$ by Proposition \ref{prop:charts2Side}. For 
the second term, we bound
\begin{align*}
\| df^{-{n_0}}_{x} - df^{-{n_0}}_{y} \| & \leq \sum_{j = 0}^{n_0-1} \big\| df^{-1}_{x_{-(n_0-1)}} \circ \cdots \circ df^{-1}_{x_{-(j+1)}} \circ \big( 
df^{-1}_{x_{ - j}} - df^{-1}_{y_{-j}} \big) \circ df^{-1}_{y_{ - (j - 1)}} \circ \cdots \circ df^{-1}_{y_{ 0}} \big\| \\
& \leq \sum_{j = 0}^{n_0-1} \bigg( \prod_{\substack{m = 0 \\ m \neq j}}^{n_0-1} \| df^{-1}_{\omega_{-m}} \| \bigg) \cdot \| d^2 f_{\omega_{-j}}^{-1} \| \cdot d(x_{ - j}, y_{ - j}) \\
& \leq n_0 l_0^{2 n_0} e^{(2 n_0^2 + n_0) \d_2} d(x, y) =: C_{n_0} d(x, y) \, .
\end{align*}
Here, for $\omega \in \Omega$ we write $\| df^{-1}_\omega\|, \| d^2 f^{-1}_\omega\|$ for uniform norms over $M$ and have used repeatedly the bound
$\|df^{-1}_{\omega_{-i}}\|, \| d^2 f^{-1}_{\omega_{-i}}\| \leq e^{i \d_2} l_1(\uo) \leq e^{i \d_2} l_0$. We now compute a lower bound on $\| df^{-n_0}_{y} v\|$: 
\begin{align*}
\| df^{-{n_0}}_{y} v \| & \geq \| df^{-{n_0}}_{y} v^s \| - \| df^{-{n_0}}_{y} \hat v\| \\
& \geq  l_0^{-1} e^{- {n_0}(\d_0 + \d_2)} \| v^s\| -  l_0^2 e^{- {n_0}(\lambda - \d_2)} \, ,
\end{align*}
having used the estimate $\| \hat v \| \leq \| \pi^{u}_{(y, \uo')} \| \leq  l(y, \uo') \leq  l_0$. 
Collecting,
\[
\| v^s\| \leq ( l_0^2 +  l_0^3) e^{- n_0(\lambda - 2 \d_2 - \d_0)} +  l_0 e^{n_0 (\d_0 + \d_2)} C_{n_0} d(x, y) \, .
\]
Fix $n_0=n_0(\epsilon, l_0)$ large enough that the first term is $< \epsilon/2$. Now, choose $\eta = \eta(\e, l_0, n_0)$ so
that the second term is $< \e/2$ when $d(x, y) < \eta$. 
\end{proof}

\subsection{Graph transforms and unstable manifolds}

We begin by recalling the definition of local unstable manifolds.

\begin{prop}[Unstable Manifold Theorem]\label{prop:localUMfld}
Let $\Gamma$ be as in Proposition \ref{prop:charts2Side}, and let $\d > 0$ be sufficiently small. Then there is a unique family of measurably-varying  maps $\{ g_{(x, \uo)} :  B^u(\d l(x, \uo)^{-1}) \to \R^{cs}\}_{(x, \uo) \in \Gamma}$ and 
a constant $C>0$ such that
\[
g_{(x, \uo)}(0) = 0 \quad \text{ and } \quad 
\tilde f_{(x, \uo)} (\graph g_{(x, \uo)}) \supset \graph g_{\tau(x, \uo)} 
\]
for every $(x, \uo) \in \Gamma$. Moreover,
\begin{itemize}
\item[1. ] $g_{(x, \uo)}$ is $C^{1 + \Lip}$, and $(dg_{(x, \uo)})_0 = 0$;
\item[2. ] $\Lip(g_{(x, \uo)}) \leq 1/10$, \ \ \ $\Lip(d g_{(x, \uo)}) \leq C l(x, \uo)$; \ and 
\item[3. ] if $z_1, z_2 \in \tilde f^{-1}_{(x, \uo)} (\graph g_{\tau(x, \uo)})$, then
\[
| \tilde f_{(x, \uo)} z_1 - \tilde f_{(x, \uo)} z_2 | \geq (e^\lambda - \d) |z_1 - z_2| \, .
\]
\end{itemize}
\end{prop}

We write $W^u_{(x, \uo), \d} = 
\Phi_{(x, \uo)} (\graph g_{(x, \uo)})$, where $g_{(x, \uo)}$ is as
above. The sets $W^u_{(x, \uo), \d}$ are the \emph{local unstable manifolds} at
$(x, \uo)$. 
The \emph{global unstable manifold}
\[W^u_{(x, \uo)} = \bigcup_{n \geq 0} f^n_{\theta^{-n} \uo} W^u_{\tau^{-n} (x, \uo),\d} \, , \] 
 is an immersed submanifold of $M$.

Since Proposition \ref{prop:localUMfld} is well-known, we omit its full proof.
We do, however, note that it can be proved by graph transform techniques, some details of which
we recall here for later use. For a Lipschitz continuous map $g : B^u(\d l(x, \uo)^{-1}) \to \R^{cs}$, we define
the \emph{graph transform} $\Tc_{(x, \uo)} g$ of $g$, when it exists, to
 be the mapping $\Tc_{(x, \uo)} g : B^u(\d l(\tau(x, \uo))^{-1}) \to \R^{cs}$ for which
 \[
 \tilde f_{(x, \uo)} \graph g \supset \graph \Tc_{(x, \uo)} g
 \]

The following Lemma summarizes what we will need about  
$\Tc_{(x, \uo)}$:  

\begin{lem}\label{lem:graphTransform} 
Let $\d > 0$ be sufficiently small. 
\begin{itemize}
\item[(a)] Let $g : B^u(\d l(x, \uo)^{-1}) \to \R^{cs}$ be such that 

(i) $g$ is $C^{1 + \Lip}$ with $\Lip(g) \leq 1/10$ and

(ii) there exists $z \in \graph g$
such that  $z \in B(\frac12 \d l(x, \uo)^{-1}) \cap \tilde f_{(x, \uo)}^{-1}  
B(\frac12 \d l (\tau(x,\uo))^{-1})$. 

\noindent
Then $\Tc_{(x, \uo)} g : B^u(\d l(\tau (x, \uo))^{-1}) \to \R^{cs}$ exists,
with $\graph \Tc_{(x, \uo)} g \subset B(\d l(\tau(x, \uo))^{-1})$. Moreover,
$\Tc_{(x, \uo)} g$ is $C^{1 + \Lip}$ and satisfies $\Lip(\Tc_{(x, \uo)} g) \leq 1/10$. 

\item[(b)] Let $g^1, g^2$ be as in (a),  with (ii) replaced by $g^i(0)=0$. 
Then $\Tc_{(x, \uo)} g^i(0)=0$, and 
\[
| \Tc_{(x, \uo)} g^1 - \Tc_{(x, \uo)} g^2 |' \leq c | g^1 -  g^2 |' \, ,  
\]
where 
\[
| h |' := \sup_{u \in B^u(\d l(x, \uo)^{-1}), u \ne 0} \frac{|h(u)|}{|u|} \, 
\]
and $c \in (0,1)$ is a constant independent of $(x, \omega)$.
\end{itemize}
\end{lem}

%
%
%
%
%
 
 Lemma \ref{lem:graphTransform} is standard and its proof is omitted.

Next, we recall the following distortion estimate along unstable leaves.
As is well-known, the quality of such distortion estimates is a function only of the uniformity estimates at the  end of the trajectory, as we describe below.

\begin{lem}\label{lem:distortionHatSc}
Let $\d > 0$ be sufficiently small. Then for any $l_0 > 1$, there exists $D = D(l_0) > 0$ for which the following holds. Let $(x, \uo) \in \Gamma$ be such that $l(x, \uo) \leq l_0$, and let $p_1, p_2 \in W^u_{(x, \uo), \d}$. For arbitrary $n \geq 1$, write $W = W^u_{\tau^{-n}(x, \uo), \d}$. 
Then,
\[
\bigg| \log \frac{\det (df^n_{\theta^{-n} \uo} |_{T W}) ( f^{-n}_\uo p_1)}{ \det (df^n_{\theta^{-n} \uo} |_{T W}) ( f^{-n}_\uo p_2)} \bigg| \leq D( l_0) d(p_1, p_2) \, .
\]
\end{lem}

As before, to control the possible unboundedness of the sequence $\| f_{\omega_n}^\pm\|_{C^2}$,
we incorporated $l_1$ into the definition of $l$ as in Sect. 2.1. Details are left to the reader.

{\subsection{Stacks of unstable leaves}}

All $\mu_\uo$-typical points have $W^u$-leaves passing through them, so $\mu_\uo$ itself can be thought of as being supported on a union of $W^u$-leaves. 
At issue is whether the conditional measures of $\mu_\uo$ on these leaves are
in the Lebesgue measure class. 
One way to articulate these ideas geometrically is to group nearby 
$W^u$-leaves into a \emph{stack}. We introduce here some language that 
will be useful later on.

\medskip
\newcommand{\Dom}{\operatorname{Dom}}

\noindent {\it Switching axes. }
 Let $x,y \in M$ be nearby points,
  and let $T_x M  = E_x \oplus F_x, T_y M = E_y \oplus F_y$ 
be such that $d_H(E_x, E_y) , d_H(F_x, F_y)$ $ \ll 1$. For $z = x,y$, write $\pi_z : T_z M \to E_z$
for the projection parallel to $F_z$.
 Given a mapping $\phi_y : \Dom(\phi_y) \to F_y$ defined on a set $\Dom(\phi_y) \subset E_y$,
 we write $\phi_y^x : \Dom(\phi_y^x) \subset E_x \to F_x$ for the mapping, 
 if it can be uniquely defined, such that
 \[
 \exp_x \graph \phi_y^x = \exp_y \graph \phi_y \, .
 \]
Below we give a condition to guarantee the well-definedness of $\phi_y^x$.
$\calLip(\cdot)$ refers to 
Lipschitz constants with respect to the norms $\|\cdot\|$.

\medskip
\begin{lem}\label{lem:chartSwitch11}
Given $L > 1, \rho > 0$, there exist $\e_1 = \e_1(L, \rho) \ll \rho$ and  $\e_2 = \e_2(L)$ such that the
following holds. Assume that 
\begin{itemize}
\item[(i)] $\| \pi_x \|, \| \pi_y \| \leq L$; 
\item[(ii)] \begin{itemize}
	\item[(1)] $d(x,y) < \e_1$; 
	\item[(2)] $d_H(E_x, E_y), d_H(F_x, F_y) < \e_2$; and
	\end{itemize}
\item[(iii)] $\phi_y : E_y(2 \rho) \to F_y(\frac12 \rho)$ is a Lipschitz mapping with $\calLip(\phi_y) \leq 1/10$.
\end{itemize}
Then $\phi_y^x$ exists, is defined on $E_x(\rho)$ with $\phi_y^x \big( E_x(\rho)\big) \subset F_x(\rho),$ and has $\calLip(\phi^x_y) \leq 2 \calLip(\phi_y)$. Moreover, $\exp_y \graph \phi_y|_{E_y(\frac12 \rho)} \subset \exp_x \graph \phi_y^x|_{E_x(\rho)}$. 
\end{lem}

 The proof is {straightforward} and is left to the reader 
 (for more detail, see Sect. 5 of \cite{blumenthal2017entropy}). 

For $l_0>1$, let 
$$\Gamma_{l_0, \uo}  = \{x \in M: (x, \uo) \in \Gamma \text{ and }  l(x, \uo) \le l_0\}\ ,
$$
and let $\overline A$ denote the closure of the set $A$. 


\begin{lem}\label{lem:uStack2Side}
Let $\d > 0$ be as in Proposition \ref{prop:localUMfld}, and
let $l_0>1$ be fixed. For all  $r=r(l_0,\d)$ and $\e=\e(l_0, \d,r)$
sufficiently small (in particular, $\e \ll r$), the following holds.
Fix $\uo \in \Omega^\Z$ and $x_* \in \Gamma_{l_0, \uo}$.
View $x_*$ as a reference point, and  write 
$E_*^{u/cs} (r) = E^{u/cs}_{(x_*, \uo)}(r)$ and
$E_*(r) = E_*^u(r) + E^{cs}_*(r)$. Then 
\begin{itemize}
\item[(a)] for each $x \in \overline{\Bc( x_*, \e)} \cap \Gamma_{l_0, \uo}$, there is a $C^{1 + \Lip}$ map
$\gamma_x : E_*^u(r) \to E_*^{cs}(r)$ with $\calLip(\gamma_x) \leq 1$ such 
 that the connected component
of $W^u_{(x, \uo), \d} \cap \exp_{x_*}( E_*(r))$ containing $x$ coincides with $\exp_{x_*} \graph \gamma_x$;
\item[(b)] the assignment $x \mapsto \gamma_x$ varies continuously in the uniform norm on $C(E_*^u(r), E_*^{cs}(r))$ as $x$ varies in $\overline{\Bc( x_*, \e)} \cap  \Gamma_{l_0, \uo}$.
\end{itemize}
\end{lem}

\begin{proof} For $x \in  \Gamma_{l_0, \uo}$, define $\check g_{(x, \uo)} = L_{(x, \uo)} \circ g_{(x, \uo)} \circ L_{(x, \uo)}^{-1}$ 
{where $g_{(x, \uo)}$ is as in Lemma  \ref{prop:localUMfld} and
$L_{(x, \uo)}$ is as in Proposition \ref{prop:charts2Side}, so that $\check g_{(x, \uo)}$
is a graphing map from $E^u_{(x, \uo)}$ to $E^{cs}_{(x, \uo)}$ in $T_x M$.}

For (a), we use Lemma \ref{lem:chartSwitch11} to change the axes of $\check g_{(x, \uo)}$ from $E^{u/cs}_{(x, \uo)}$ to $E_*^{u/cs}$: Item (i) in Lemma \ref{lem:chartSwitch11} 
is satisfied with $L = l_0$, and (ii) follows from the continuity of $E^{u/cs}$ subspaces through points of $ \Gamma_{l_0, \uo}$ (Proposition \ref{prop:ctyEucs}). To arrange for (iii),
 observe that our control on $\Lip(g_{(x, \uo)})$ is only in the adapted norm $|\cdot|$, not the Riemannian metric $\| \cdot \|$ on $T_x M$;
 generally we have only the very poor bound $\calLip(\check g_{(x, \uo)}) \leq l_0 \Lip(g_{(x, \uo)})$. This is remedied by
 truncating the domain of $\check g_{(x, \uo)}$ to $E^u_{(x, \uo)}(2 r)$, where $r > 0$ is chosen sufficiently small
 so that (i) $\check g_{(x, \uo)}$ is defined on $E^u_{(x, \uo)}(2 r)$, and (ii) $\calLip(\check g_{(x, \uo)}|_{E^u_{(x, \uo)}(2 r)} ) \leq 1/10$. For the latter, we take advantage of the fact that $(d g_{(x, \uo)})_0 = 0$ and our
control on $\Lip(d g_{(x, \uo)})$. A simple computation implies this
can be arranged by taking $r = \min\{(20 C l_0^3)^{-1}, \frac12 \d l_0^{-2}\}$, with $C, \d$ as in Proposition \ref{prop:localUMfld}.

(b) follows from the continuity of $E^u$ subspaces (Proposition \ref{prop:ctyEucs})
and the contraction estimate for the graph transform (Lemma \ref{lem:graphTransform}). 
Details are left to the reader; a similar argument is carried out in the proof of Proposition \ref{lem:ctyInX} in Section 5 of this paper; see also Lemma 5.5 in \cite{blumenthal2017entropy}. 
\end{proof}

We refer to 
\[
\Sc_\uo := \bigcup_{x \in \overline{\Bc( x_*, \e)}\cap  \Gamma_{l_0, \uo}} \xi(x) \,  \quad \text{ where } \quad \xi(x) := \exp_{x_*}( \graph \gamma_x) \, 
\]
as a \emph{stack of unstable leaves} through $\overline{\Bc(x_*, \e)}\cap  \Gamma_{l_0, \uo}$. 

\subsection{Main proposition and proof of Main Theorem}

For $\uo \in \Omega^{\Z}$ let us write $\mu^n_\uo = (f^n_{\theta^{-n} \uo})_* \mu$, recalling that $\mu^n_\uo \to \mu_\uo$ weakly with $\P$-probability 1 by Lemma \ref{lem:invariantMeasures}. Our plan is
to fix $\uo$, and track the orbits of a small fraction of $\mu$-typical points from time 
$-n$ to time $0$ with the aid of Lyapunov charts. We will  
show that their images at time $0$ are increasingly aligned with local unstable manifolds, and that as $n \to \infty$, the weak limits of this sequence of small `pieces' of $\mu^n_\uo$ possess the SRB property. This is summarized in the following Main Proposition of this paper. The notation is as in Sect. 2.2.

\begin{prop}[Main Proposition] \label{prop:main} {For all sufficiently large} $l_0>1$, there is a positive 
$\P$-measure set of $\uo$ and a small constant $c>0$ for which the following hold.
On each $ \Gamma_{l_0, \uo}$, there is a stack $\Sc_\uo$
of local unstable manifolds with the following properties:
\begin{itemize}
\item[(a)] a fraction $\ge c$ of $\mu_\uo$ is supported on $\Sc_\uo$, i.e., 
$\mu_\uo = \nu_1 + \nu_2$ where $\nu_1, \nu_2$ are both positive measures,
and $\nu_1$ is supported on $\Sc_\uo$ with $\nu_1(\Sc_\uo) \ge c$.
\item[(b)] Let $\Xi$ be the partition of $\Sc_\uo$ into unstable leaves. Then
the conditional probabilities of $\nu_1$ on elements of $\Xi$ are absolutely 
continuous with respect to Lebesgue measure, with densities uniformly bounded above and below.
\end{itemize}
\end{prop}

\medskip
The proof of Proposition \ref{prop:main} will occupy the rest of this paper.
We first complete the proof of the Main Theorem assuming this result.

Let $f: M \circlearrowleft$ be a (single) diffeomorphism
preserving a probability measure $\sigma$
with at least one positive Lyapunov exponent $\sigma$-a.e..
We say that a measurable partition $\eta$ of $M$ 
is \emph{subordinate to unstable manifolds} if for $\sigma$-a.e. $x$, 
$\eta(x)$, the element of $\eta$ containing $x$, is a relatively compact subset of $W^u(x)$
and contains an open neighborhood of $x$ in $W^u(x)$. To say that $\sigma$ is an
SRB measure is equivalent to saying that its conditional measures on the elements of $\eta$ 
are absolutely continuous with respect to the Riemannian measures on unstable
manifolds (see, e.g., \cite{ledrappier1985metric} for details).


These ideas extend readily to RDS. We say a partition $\eta$ of 
$M \times \Omega^\Z$ is {\it subordinate to unstable manifolds} if for $\mu^*$-a.e.
$(x,\uo)$, $\eta(x,\uo)$ is a relatively compact subset of the global unstable manifold $W^u_{(x,\uo)}$ 
(so in particular 
$\eta(x,\uo) \subset M \times \{\uo\}$) and it contains an open neighborhood of $x$ in $W^u_{(x,\uo)}$. 
The definition
of {\it random SRB measures} in Definition \ref{defn:SRB}
is equivalent to $\mu^*$ having absolutely continuous conditional measures on
elements of $\eta$.

\begin{proof}[Proof of Main Theorem assuming Proposition \ref{prop:main}] Let $\eta$ be 
a partition of $M \times \Omega^\Z$ subordinate to unstable manifolds, and
let $\mu^*_T$ be the quotient measure of $\mu^*$ on $(M \times \Omega^\Z)/\eta$. 
For a.e. element $\alpha$ of $\eta$, let $m_\alpha$ 
denote the Riemannian measure on $\alpha$. We define a (possibly sigma-finite) measure 
$\nu$ on $M \times \Omega^\Z$ by letting
$$
\nu(A) = \int m_\alpha(A \cap \alpha) \ d \mu^*_T(\alpha)
$$
for every Borel set $A \subset M \times \Omega^\Z$, and decompose $\mu^*$
into an absolutely continuous part $\mu^*_{\rm ac}$ and a singular part 
$\mu^*_{\perp}$ with respect to $\nu$. Since $\mu^*_{\rm ac}$ is preserved by
$\tau$, and $(\tau, \mu^*)$ is ergodic, we have either $\mu^*_{\rm ac}(M \times \Omega^\Z)=1$, in which case $\mu^*$ is SRB,  
or $\mu^*_{\rm ac}(M \times \Omega^\Z)=0$. 

Assume, to derive a contradiction, that
$\mu^* = \mu^*_\perp$. By definition, there exists a Borel set $A \subset M \times \Omega^\Z$ with
$\nu(A^c) = 0$ and $\mu^*_\perp(A) = 0$. In particular, $m_\alpha(A^c \cap \alpha) = 0$
and $\mu^*_\a(A \cap \alpha) = 0$ for $\mu^*_T$-almost every $\a \in \eta$.
We conclude that for such $\a$, $\mu^*_\a$ and $m_\a$ are mutually singular. This
contradicts Proposition \ref{prop:main}, which implies that for a $\mu^*_T$-positive measure
set of $\a \in \eta$, we have that $\mu^*_\a$ has a nontrivial absolutely continuous component.

We conclude that $\mu^*_{\rm ac}(M \times \Omega^\Z) > 0$, hence $\mu^* = \mu^*_{\rm ac}$ and
the proof is complete.
\end{proof}

\section{New chart systems and iterated graph transforms}

As explained in the Introduction, our plan is to realize $\mu^*_{\uo}$ as
the limit of $(f^n_{\theta^{-n}\uo})_* \mu$ as $n \to \infty$. In this section, 
we begin to prepare for this pushing-forward process, with the following 
simplifications:
(i) we will start from time $0$ rather than time $-n$, i.e., we will consider
$(f^n_{\uo^+})_*\mu, n=1,2, \dots,$ for some $\uo^+ \in \Omega^{\Z^+}$;
(ii) we will consider pushing forward $\mu$ near one $(x,\uo^+)$ at a time; and 
(iii) we will push forward graphs of functions 
rather than $\mu$. Later on, we will disintegrate
$\mu$ onto graphs of this type to be pushed forward.

\subsection{A new chart system}

{We would like to have charts 
defined at $( \mu \times \mathbb P^+)$-a.e. $(x,\uo^+)$,} so we can push forward
small pieces of graphs transversal to $E^{cs}$ in the chart at $x$. 
Adapted charts for one-sided RDS have been constructed before 
(see, e.g., Chapter III of \cite{liu2006smooth}), but to the authors' knowledge, there
are no existing constructions in the literature that are suitable for our purposes; see
the discussion following Proposition \ref{prop:chartsMeasPast}.

The construction we present here proceeds roughly as follows: since $\mu$-a.e. $x \in M$ is generic with 
respect to $\mu_\uo$ for some $\uo \in \Omega^\Z$, one may associate to 
 $( \mu \times \P^+)$-a.e.$(x,\uo^+)$ a choice of $\uo \in \Omega^\Z$
that (i) agrees with $\uo^+$ on its $\Omega^{\Z^+}$-coordinate and for which (ii) $(x,\uo) \in \Gamma$ where $\Gamma$
is as in Proposition \ref{prop:charts2Side}. We may then equip $(x,\uo^+)$ with the chart at $(x,\uo)$ from 
Proposition \ref{prop:charts2Side}. This is essentially how we will proceed, but first we need to take care of measurability issues.
Given a Borel measurable set $\Theta \subset M \times \Omega^\Z$, let 
$\Theta^+$ denote its projection to $M \times \Omega^{\Z^+}$.  

\begin{lem} \label{lem:measPastChoice}
Given any Borel set $\Theta \subset M \times \Omega^\Z$  that is 
a countable union of compact subsets, there is a Borel measurable function 
$$\hat \uo^- : \Theta^+ \to \Omega^{\Z^-}
$$ 
with the property that for any $(x, \uo^+) \in \Theta^+$, we have 
$$(x, \hat \uo(x, \uo^+)) \in \Theta \, , \qquad \mbox{where} \qquad  
\hat \uo(x, \uo^+) := \big( \hat \uo^-(x, \uo^+) , \uo^+\big) \in \Omega^\Z .
$$ 
\end{lem}

Lemma \ref{lem:measPastChoice} is a direct application of the measurable selection criterion in Lemma \ref{lem:measSelection}.
We will explain in the Appendix how Lemma \ref{lem:measSelection} can be deduced from a well known
result.

\begin{lem}\label{lem:measSelection}
Let $X,Y$ be Polish spaces. Let $G \subset X \times Y$ be a compact subset and set $G_X$ to be the projection of $G$ onto $X$. Then, there exists a Borel measurable mapping $\psi : G_X \to Y$ with the property that for any $x \in G_X$, we have $(x, \psi(x)) \in G$.
\end{lem}

We apply Lemma \ref{lem:measPastChoice} to $\Theta = \Gamma$ where 
$\Gamma$ is in Proposition \ref{prop:charts2Side}, noting that
 (perhaps diminishing 
$\Gamma$ by a $\mu^*$-null set), $\Gamma$ can be represented
as a countable union of compact sets by the inner regularity of $\mu^*$.
We then obtain $\Gamma^+ \subset M \times \Omega^{\Z^+}$ and 
$\hat \uo^-: \Gamma^+
\to \Omega^{\Z^-}$, i.e., to each $(x,\uo^+) \in \Gamma^+$, we associate 
in a measurable way a ``past" $\hat \uo^- = (\hat \omega_n)_{n \leq 0}$ so that $(x,\hat \uo) \in \Gamma$. 
Recall that $E^{cs}_{(x,\uo^+)}$, which depends only on future iterates, 
is well defined but without a past there is no intrinsic notion of 
$E^u_{(x,\uo^+)}$. We now define
$\hat E^u_{(x,\uo^+)} := E^u_{(x,\hat \uo)}$, and let 
$\hat \Phi_{(x,\uo^+)} := \Phi_{(x,\hat \uo)}$ be a chart at $(x,\uo^+)$,
$\Phi_{(\cdot, \cdot)}$ as in Proposition \ref{prop:charts2Side}. 


To each $(x,\uo^+) \in \Gamma^+$, we introduce next a sequence of charts $\{\hat \Phi^{(k)}_{(x,\uo^+)}, k=0,1,2,
\dots\}$ along the $\tau^+$-orbit of $(x,\uo^+)$, with $\hat \Phi^{(0)}_{(x,\uo^+)}=
\hat \Phi_{(x,\uo^+)}$.

\begin{prop}\label{prop:chartsMeasPast} Let $(x,\uo^+) \in \Gamma^+$, and let $(x, \hat \uo) \in \Gamma$
be given by Lemma \ref{lem:measPastChoice}. Then for each $k=0,1,2,\dots$, this induces at 
$(\tau^+)^k(x,\uo^+)$
the splitting 
$$
T_{f^k_{\uo^+}x}M = \hat E^{u, (k)}_{(x,\uo^+)} \oplus E^{cs, (k)}_{(x,\uo^+)} \qquad
\mbox{where} \qquad 
\hat E^{u, (k)}_{(x,\uo^+)} := E^u_{\tau^k(x,\hat \uo)}\ , \ \ 
E^{cs, (k)}_{(x,\uo^+)} := E^{cs}_{\tau^k(x,\hat \uo)}\ .
$$
We also define for each $k$
$$\hat l^{(k)}_{(x,\uo^+)} = l(\tau^k(x,\hat \uo))$$ where $l$ is
as in Proposition \ref{prop:charts2Side}, and define at $(\tau^+)^k(x,\uo^+)$ a 
chart given by
$$
\hat \Phi^{(k)}_{(x,\uo^+)} = \Phi_{\tau^k(x,\hat \uo)} =  \exp_{f^k_{\uo^+} x} \circ \hat L^{k}_{(x, \uo^+)} \, , \quad 
 \quad \hat L^{(k)}_{(x, \uo^+)} := L_{\tau^k(x, \hat \uo)}  
$$
where $\Phi_{(\cdot, \cdot)}, L_{(\cdot, \cdot)}$ are as in Proposition \ref{prop:charts2Side}. Then for each $k$, $$     
(x, \uo^+) \mapsto \hat E^{u,(k)}_{(x, \uo^+)}, \quad E^{cs,(k)}_{(x, \uo^+)}, \quad
\hat \Phi^{(k)}_{(x,\uo^+)}\ ,\quad \hat l^{(k)}_{(x, \uo^+)} \ , \quad  \hat L^{(k)}_{(x, \uo^+)}
$$
are measurable functions, and the properties of the 
maps $\hat f^{(k)}_{(x, \uo^+)} := \tilde f_{\tau^k(x, \hat \uo)}$ along
this sequence of charts are, by construction, the same as in Proposition \ref{prop:charts2Side}.
\end{prop}
 
Notice that we have associated to each $(x,\uo^+) \in \Gamma^+$ a sequence
of charts along its $\tau^+$-orbit by applying the measurable
selection lemma {\it once}, to the point $(x,\uo^+)$. Since in general 
 $\hat \uo(\tau^+(x, \uo^+)) \neq \theta \hat \uo(x, \uo^+)$, the sequence of charts
associated to $(x,\uo^+)$ shifted forward once is different from the the sequence associated 
to the point $\tau^+(x,\uo^+)$. That is to say, these sequences of charts are
dependent on the initial points $(x,\uo^+)$, and we have stressed that by
putting $(x,\uo^+)$ in the subscripts of $\hat E^{u,(k)}_{(x, \uo^+)}, \hat \Phi^{(k)}_{(x,\uo^+)}$ etc. even though these objects are attached to the point 
$(\tau^+)^k(x, \uo^+)$.

With regard to differences with existing constructions of one-sided 
charts (as used in, e.g., Chapter III of \cite{liu2006smooth}), previously constructed
charts are only guaranteed to 
have size at least $C^{-1} e^{- n \d_2}$ at time $n$ where $C$ depends on the initial point. In contrast, the construction in Proposition \ref{prop:chartsMeasPast} has the property that the chart size at time $n$ is $\sim (\hat l^{(n)}_{(x, \uo^+)})^{-1} = l(\tau^n(x, \hat \uo))^{-1}$. For $(\mu \times \P^+)$-typical $(x, \uo^+)$, 
these chart sizes are guaranteed to rise above some minimum size for 
infinitely many $n$, a property crucial for our constructions in Sections 4--6.
\smallskip

In the rest of this section, the selection function 
$\hat \uo^- : \Gamma^+ \to \Omega^{\Z^-}$
given by Lemma \ref{lem:measPastChoice} is fixed, and the chart system in use will be 
$$
\{\hat \Phi^{(k)}_{(x,\uo^+)}, k=0,1,2, \dots | (x,\uo^+) \in \Gamma^+\}\ .
$$

\noindent 
{\bf Uniformity sets}

\medskip
For $l_0 > 1$ and $k \ge 0$, we let 
$$
\Gamma^{+, (k)}_{l_0} = \{(x,\uo^+) \in
\Gamma^+ \ | \ {\hat l^{(k)}_{(x,\uo^+)}} \le l_0\}.
$$
These are clearly versions of the uniformity sets described in Section 2.2.

Observe that since $E^{cs}$ subspaces depend only on the future, they have
no dependence on the measurable selection made at time $0$.
As in Proposition \ref{prop:ctyEucs}(a), it follows that 
$E^{cs, (k)}_{(x, \uo^+)} = E^{cs}_{(\tau^+)^k(x, \uo^+)}$
varies continuously across points of $(\tau^+)^k(\Gamma^{+, (k)}_{l_0})$. 
The situation for $\hat E^u$ is different, and the following observations are crucial:

\begin{rmk}\label{rmk:ctyEu} \
\begin{itemize}
\item[(a)] We claim that the subspaces $\hat E^{u}_{(x, \uo^+)}$ and $E^{cs}_{(x', \uo^+)}$ are uniformly separated
when $(x, \uo^+)$, $(x', \uo^+) \in \Gamma^{+}_{l_0}$ are sufficiently close. While we do
not have that $\hat E^u_{(x,\uo^+)}$ and $\hat E^u_{(x',\uo^+)}$ are close,			
we have $\| \hat \pi_{(x, \uo^+)}^{u} \| \le l_0$ where $\hat \pi^{u}_{(x, \uo^+)} : T_{x}M \to \hat E^{u}_{(x, \uo^+)}$ is the projection parallel to $E^{cs}_{(x, \uo^+)}$. This together with the continuity of
$E^{cs}$ as discussed above implies uniform separation, as claimed.
\item[(b)] In light of Proposition \ref{prop:ctyEucs}(b), 
			 the dependence of $\hat E^{u, (k)}$ on the measurable selection becomes weaker and weaker as $k$ is increased; that is,
			 although $\hat E^{u, (k)}$ does not vary continuously, nearby $\hat E^{u, (k)}$ subspaces become very well aligned for $k \gg 1$.	
\end{itemize}			
\end{rmk}

\subsection{Transforms of graphs (with possibly large slopes)}

Graph transforms {were discussed in Sect. 2.3; 
what is new here} is that we have to consider graphs with possibly large 
though uniformly bounded slopes, the reason being the observation in Remark \ref{rmk:ctyEu}(a).
This subsection gives {\it a priori} bounds for a single step of the graph transform.
Let $(x, \uo^+) \in \Gamma^+$ be fixed; we will omit mention of $(x, \uo^+)$ 
in the remainder of Section 3, writing $\hat \Phi^{(k)} = \hat \Phi^{(k)}_{(x, \uo^+)}, 
\hat f^{(k)} = \hat f^{(k)}_{(x, \uo^+)}, \hat l^{(k)} = \hat l^{(k)}_{(x, \uo^+)}$ etc.
For $k \ge 0$ and $\d \in (0,1)$, let 
$$g :  B^u(\d (\hat l^{(k)})^{-1}) \to \R^{cs}$$ 
be a mapping. The \emph{graph transform} $\Tc^{(k)} g = \Tc^{(k)}_{(x, \uo^+)} g$, 
if it {is defined}, is a mapping 
$$\Tc^{(k)} g :  B^u(\d' (\hat l^{(k+1)})^{-1}) \to \R^{cs} 
\qquad \mbox{with} \qquad \hat f^{(k)} \big( \graph g \big) \supset \graph \Tc^{(k)} g 
$$
for some $\d' \in (0,1)$. We write $K_0 = \Lip(g)$ for the Lipschitz 
constant of the initial graph $g$.
The following lemma does not distinguish between large
and small $K_0$.

\begin{lem}\label{lem:inclinationGT}
For any $K_0 > 0$, there exist constants $r_1 = r_1(\lambda, \d_0, K_0) \in ( 0,1), C_1 = C_1(\lambda, \d_0, K_0) , C_2 = C_2(\lambda, \d_0, K_0)$ 
such that the following holds 
when $\d < r_1$. Let $k \geq 0$, {$\rho \in (0,K_0^{-1}]$,} and let 
$$g :  B^u(\rho \d (\hat l^{(k)})^{-1}) \to \R^{cs}
$$ 
be a $C^{1 + \Lip}$ map for which 
(i) $g(0) = 0$, \ (ii) $\graph g \subset  B(\d (\hat l^{(k)})^{-1})$, \ and 
(iii) $\Lip(g) \leq K_0$. 

\noindent
Then, the graph transform 
$$\Tc^{(k)} g: B^u(\rho' \d (\hat l^{(k + 1)})^{-1}) \to \R^{cs}
\qquad \mbox{with} \quad \rho' = \min\{\rho e^{\lambda/2}, 1\} $$ 
{ is defined,} and is a $C^{1 + \Lip}$ map for which 
(i') $\Tc^{(k)} g(0) = 0$, \ 
{(ii') $\graph \Tc^{(k)} g \subset B(\d (\hat l^{(k)})^{-1})$,} \ 
(iii') $\Lip(\Tc^{(k)} g) \leq K_0 e^{- \lambda / 2}$, \ and
(iv') \begin{gather*}
|(d \Tc^{(k)} g)_0| \leq e^{- \lambda/2} | (d g)_0| \, ,  \\
\Lip(d \Tc^{(k)} g) \leq C_1 \hat l^{(k)} + C_2 \Lip(d g) \,. 
\end{gather*}
\end{lem}

\begin{proof} For short, let us write $\hat g = \Tc^{(k)} g$, $\hat f = \hat f^{(k)}$, $\hat l =
\hat l^{(k)}$.

Let 
\[
\Cc = \Cc(K_0) = \{ u + v : u \in \R^u, v \in \R^{cs} \, , \text{ and } |v| \leq K_0 |u| \} \, .
\]
Observe that $d \hat f_0$ maps $\Cc$ strictly into its interior. Let
$r_1 = r_1(\lambda, \d_0, K_0) > 0$ be small enough that for $\d \in (0, r_1)$,  
$$
d \hat f_z \Cc \subset  {\Cc(K_0 e^{-\lambda/2})} \quad \text{ for all } z \in  B(\d \hat l^{-1}) \, .
$$
More precisely, if $w = u+v \in \Cc \subset \R^u \oplus \R^{cs}$ 
and $d\hat f_z(w)= u'+v' \in \R^u \oplus \R^{cs}$, then 
\begin{eqnarray*} 
|u'| & \geq & e^{\lambda} |u| - \d \max\{|u|, |v|\}  
\ \geq \ (e^{\lambda} - \d \max\{1,K_0\}) |u| \, , \\
\mbox{and} \quad 
|v'| & \leq & e^{\d_0} |v| + \d \max\{|u|, |v|\} \ \leq  \ e^{\d_0}|v| + \d\max\{1,K_0\} |u| \, ,
\end{eqnarray*}
and $u'+v' \in \Cc$ for $\d \in (0,r_1)$.

\medskip
Let now $\rho, \d$, and let $g$ be as in the hypothesis of 
Lemma \ref{lem:inclinationGT}. We let
$$\phi : B^u(\rho \d {\hat l}^{-1}) \to \R^u 
\qquad \mbox{be given by} \qquad 
\phi(u) = \pi^u \circ \hat f(u, g(u))\ ,
$$
 and define
$$
\hat g = \Tc^{(k)} g = \pi^{cs} \circ \hat f \circ (\Id \times g) \circ \phi^{-1}
$$ 
where $\Id$ refers to the identity map restricted to 
$ B^u(\rho \d {\hat l}^{-1})$.
As the existence of $\hat g$ and its first derivative properties follow largely
from standard arguments involving the invariant cones condition above, 
we leave them to the reader,
providing below only the bound for $\Lip(d \hat g)$.

\medskip
Let $\hat u_1, \hat u_2 \in B^u(\rho \d ({\hat l}^{(k + 1)})^{-1})$ with $u_i = \phi^{-1}(\hat u_i), i = 1,2$. Then,
\begin{align}\label{eq:developLipEst}\begin{split}
| d \hat g_{\hat u_1} - d \hat g_{\hat u_2}|  \leq \, & 
| \pi^{cs} \circ \big( d \hat f_{(u_1, g(u_1))} - d \hat f_{(u_2, g(u_2))} \big) | \cdot |\Id + d g_{u_1}| \cdot |d \phi^{-1}_{\hat u_1}|   \\
& + | \pi^{cs} \circ d \hat f_{(u_2, g(u_2))} (d g_{u_1} - d g_{u_2})| \cdot | d \phi^{-1}_{\hat u_1}| \\
& + | \pi^{cs} \circ d \hat f_{(u_2, g(u_2))} (\Id + d g_{u_2})| \cdot | d \phi^{-1}_{\hat u_1} - d \phi^{-1}_{\hat u_2} |\ .
\end{split}
\end{align}
We bound the last term of \eqref{eq:developLipEst} as follows:
	\begin{itemize}
\item First we have $|d\phi^{-1}_{\hat u}| \leq e^{- \lambda/2}$ for all $\hat u \in  B^u(\rho \d ({\hat l}^{(k+1)})^{-1})$.
Using the fact that $d\hat f_z = d\hat f_0 + (d\hat f_z - d\hat f_0)$
and $|d\hat f_z - d\hat f_0| \le \hat l |z| \le  \d$
for $z \in B^u(\d {\hat l}^{-1})$ by Proposition \ref{prop:charts2Side}, we have 			
\begin{align*}
|d \phi_{u_1} - d \phi_{u_2}| & \leq |d \hat f_{(u_1, g(u_1))} \circ (\Id + d g_{u_1}) - d \hat f_{(u_2, g(u_2))} \circ (\Id + d g_{u_2}) | \\
& \le |d \hat f_{(u_1, g(u_1))} - d \hat f_{(u_2, g(u_2))}| \cdot |\Id + d g_{u_1}|
+ |d \hat f_{(u_2, g(u_2))} ((\Id + d g_{u_1}) - (\Id + d g_{u_2}))|\\
& \leq {\hat l}\max\{1, K_0\} |u_1-u_2| \cdot \max\{1, K_0\} + (e^{\d_0} + \d) 
\Lip(dg) |u_1-u_2|\\
& \leq e^{-\lambda/2} \bigg( \max\{1, K_0\}^2 {\hat l} + (e^{\d_0} + \d) \Lip(d g) \bigg) |\hat u_1 - \hat u_2| 
			\end{align*}
so that 
\begin{align*}
|d \phi^{-1}_{\hat u_1} - d \phi^{-1}_{\hat u_2}| & \leq |d \phi^{-1}_{\hat u_1}| \cdot |d \phi^{-1}_{\hat u_2}| \cdot |d \phi_{\hat u_1} - d \phi_{\hat u_2}| \\
& \leq e^{-3\lambda/2} \bigg( \max\{1, K_0\}^2 {\hat l} + (e^{\d_0} + \d) \Lip(d g) \bigg) |\hat u_1 - \hat u_2|\ .
			\end{align*}		
\item Since for $w \in \R^u$ with $|w|=1$, we have
$$
| \pi^{cs} \circ d \hat f_{(u_2, g(u_2))} (w+dg_{u_2}w)| \le
e^{\d_0} K_0 + \d \max\{1,K_0\}\ ,
$$
this is an upper bound for $| \pi^{cs} \circ d \hat f_{(u_2, g(u_2))} (\Id + d g_{u_2})|$.
\end{itemize}
Finally, plugging these back into \eqref{eq:developLipEst}, we obtain
\[
| d \hat g_{\hat u_1} - d \hat g_{\hat u_2}| \leq \big( C_1 {\hat l} + C_2 \Lip(d g) \big) \cdot |\hat u_1 - \hat u_2| \, ,
\]
where
\begin{gather*}
C_1 = e^{- \lambda} \max\{1, K_0\}^2 + (e^{\d_0}K_0 + \d \max\{1, K_0\})
\max\{1, K_0\}^2 e^{- 3 \lambda / 2} \, , \\
C_2 = e^{- \lambda} (e^{\d_0} + \d) + 
(e^{\d_0}K_0 + \d \max\{1, K_0\})(e^{\d_0} + \d) e^{- 3 \lambda / 2}  \, .
\end{gather*}
\end{proof}

Lemma \ref{lem:inclinationGT} provides us with the following information: In general, $C_2>1$, 
which is not useful for controlling the growth of
$\Lip({d{\cal T}^{(k)}g})$ as we iterate the graph transform. 
However, when $K_0$ is small enough depending mostly on $\lambda$ (also 
$\delta_0$ and $\delta$), then $C_2(K_0) < 1$. We fix
$\bar K_0 \le \frac{1}{10}$ small enough that 
{$C_2(\bar K_0) e^{\d_2} < 1$, and 
write $\bar C_i = C_i(\bar K_0), i = 1,2$. Furthermore, we let }
$\bar r_1 := \bar r_1(\bar K_0)$ be small enough that on 
$B(\bar r_1(\hat l^{(k)})^{-1})$, the cones 
$\Cc(\bar K_0)$ are invariant under $d\hat f_z$.

\subsection{Iteration of graph transforms}

We now consider iterated graph transforms along the orbit of $(x,\uo^+) \in \Gamma^+$,
introducing first the following notation: Given $g$ and a sequence of numbers 
$d_0, d_1, \dots$, we say 
$$
g_{k+1} = {\cal T}^{(k)} \circ \cdots \circ {\cal T}^{(0)}g\ , \qquad k=0, 1,2, \dots,
$$
are the graph transforms of $g$ on $ B(d_k(\hat l^{(k)})^{-1})$ if
$$\graph g_0= \graph g \cap  B(d_0(\hat l^{(0)})^{-1})\ ,
$$
and for each $k \ge 0$, we let
$$
\graph g_{k+1} = \hat f^{(k)}(\graph g_k) \cap  B(d_{k+1}(\hat l^{(k+1)})^{-1})\ ,
$$ 
assuming the graph transforms above are well defined. 
{In this definition, we allow the domain 
of definition of $g_k$ to be a proper subset of $ B^u(d_k(\hat l^{(k)})^{-1})$
containing $0$ (but when we write $h:U \to V$, it will be implied that $h$
is defined on all of $U$).}

Let $\bar K_0$ and $\bar r_1$ be as in Sect. 3.2.

\begin{prop}\label{lem:inclination} Given $K_0, \lambda_0, \d_0$,
there exist $\bar C \ge 1$ (independent of $K_0$), 
$m_0=m_0(K_0)$ and {$\bar r_0=\bar r_0(K_0, m_0)>0$} 
for which the following hold. Let $r_0 \leq \bar r_0$. 
Then there exists
$m_1 \in \Z^+$ depending on $\Lip(dg)$ in addition 
to the constants above with the following properties. Let 
$$
g: B^u(r_0 (\hat l^{(0)})^{-1}) \to \R^{cs}
$$ 
be a $C^{1+{\rm Lip}}$ map with 
(i) $g(0)=0$ and (ii) $\Lip(g) \le K_0$. We let $g_k$ be the graph transforms of $g$ 
{with $d_k=r_0$ for $k \le m_0$ and $d_k=\bar r_1$ for $k > m_0$.}
Then for all $k\ge m_0 + m_1$, 
$$
g_k :  B^u(\bar r_1 (\hat l^{(k)})^{-1}) \to \R^{cs}
$$
is defined and satisfies 
$$
\Lip(g_k) \le \bar K_0 \, , \quad  | (d g_k)_0| \leq e^{- k \lambda / 2} | (d g)_0| \, , \quad \mbox{and} \quad \Lip({dg_k}) \le \bar C \hat l^{(k)}\ .
$$
\end{prop}

\begin{proof} We assume $K_0 > \bar K_0$ (omit the first part of the proof
if $K_0 \le \bar K_0$). Let $m_0$ be such that $K_0 e^{-m_0 \lambda/2} < \bar K_0$.  
Fix $\bar r_0 > 0$ sufficiently small so that for each $0 \leq k \leq m_0 - 1$,  we have for $z \in B(\bar r_0 (\hat l^{(k)})^{-1})$ that $(d \hat f^{(k)})_z \, \Cc(K_0 e^{- k \lambda / 2}) \subset \Cc(K_0 e^{- (k + 1) \lambda / 2})$
(notation as in the proof of Lemma \ref{lem:inclinationGT}). By 
the estimates in the proof of Lemma \ref{lem:inclinationGT}, the choice of $\bar r_0$
depends on $m_0, K_0$.

With $r_0 \leq \bar r_0$ now fixed and $\{ g_k\}$ the graph transform sequence
as defined in the statement, we obtain from a simple induction 
{that $g_{m_0}$ is defined on $B^u(r_0 (\hat l^{(m_0)})^{-1})$ and
 $\Lip(g_{m_0}) \leq K_0 e^{- m_0\lambda / 2} < \bar K_0$.}
Since $\bar K_0$-cones are 
preserved on charts of size $ B(\bar r_1 (\hat l^{(k)})^{-1})$, 
 $\Lip(g_k) \le \bar K_0$ will hold for $k \ge m_0$. 
Moreover, one easily checks (see (iv') in Lemma \ref{lem:inclinationGT}) that
$| (d g_k)_0| \leq e^{- k \lambda / 2} | (d g)_0|$
holds for all $k$.

Next, we grow $g_k$ so that its graph stretches all the way
across the chart, letting $m_1' = m_1'(r_0, \bar r_1)$ be such that $g_{m_0 + m_1'}$
is defined on all of $B^u(\bar r_1 (\hat l^{(m_0 + m_1')})^{-1})$.

%

It remains to bound $\Lip(dg_k)$. Let $a=\Lip({d g_{m_0}})$. Though $\Lip({d g_k})$ 
may have grown during the first $m_0$ iterates, $a$ is determined by 
$\Lip(dg), K_0, \lambda$ and $m_0$ (Lemma \ref{lem:inclinationGT}). Applying Lemma \ref{lem:inclinationGT} again repeatedly
from step $m_0$ on, we obtain 
$$
\Lip({d g_{m_0+i}}) \le \bar C_1 \{(\hat l^{(m_0 + i)}) + \bar C_2 (\hat l^{(m_0 + i-1)})
+ \bar C_2^2 (\hat l^{(m_0 + i -2)}) + \cdots + \bar C_2^i  a\}\ .
$$
Let $\bar C = 2 \bar C_1 \sum_i ( \bar C_2e^{\d_2})^i$, and choose $m_1 \ge m_1'$ 
large enough that $\bar C_1 \cdot \bar C_2^{m_1} a< \frac12 \bar C$. The desired properties are
achieved for $k \ge m_0+m_1$.
\end{proof}

For the remainder of Section 3 we fix $K_0 > 0$, 
$r_0 < \bar r_0 (K_0, m_0(K_0))$ and assume Proposition \ref{lem:inclination}
 has been applied to a fixed $C^{1 + \Lip}$ graphing function 
 $g : B^u(r_0 (\hat l^{(0)})^{-1}) \to \R^{cs}$, $\Lip(g) \leq K_0$, obtaining
the graph transform sequence $\{ g_k\}$, with all notation (e.g., $m_0, m_1$) as in the conclusions of Proposition \ref{lem:inclination}.

First, we give a distortion estimate in this setting.

\begin{lem}\label{lem:distortion}
Write $a_0 = \Lip(d g)$. Then for any $k \geq m_0 + m_1$, 
there exists a constant $D = D(K_0, a_0, {r_0}; \hat l^{(0)}, \hat l^{(k)})$ 
with the following property.
Write $\gamma_{j} = \hat \Phi^{(j)} (\graph g_{j})$ for $0 \leq j \leq k$,
and let $p_1, p_2 \in \gamma_k$. Then,
\[
\bigg| \log \frac{\det( df^k_{\uo^+} | T \gamma_{0}) \big( (f^{k}_{\uo^+})^{-1}p_1 \big)  }
{\det( df^k_{\uo^+} | T \gamma_{0}) \big( (f^{k}_{\uo^+})^{-1}p_2 \big) } \bigg| \leq D \, . 
\]
\end{lem}

{Note that $D$ does not depend on $k$ except through the value of $\hat l^{(k)}$.}

\begin{proof}
For $i = 1,2$, write $p_k^i = p_i$ and $p_0^i = (f^k_{\uo^+})^{-1} p_k^i$, and for $0 \leq j < k$ set $p_j^i = f^j_{\uo^+} p_0^i \in \gamma_j$. We decompose
\begin{gather*}
\bigg| \log \frac{\det( df^k_{\uo^+} | T \gamma_{0}) (p_0^1)}{\det( df^k_{\uo^+} | T \gamma_{0}) (p_0^2) } \bigg| 
\leq
\bigg| \log \frac{\det( df^{m_0 + m_1}_{\uo^+} | T \gamma_{0}) ( p_0^1 )}{\det( df^{m_0 +m_1}_{\uo^+} | T \gamma_{0}) ( p_0^2) } \bigg| 
+ 
\bigg| \log \frac{\det( df^{k - (m_0 + m_1)}_{\uo^+} | T \gamma_{m_0 + m_1}) (p_{m_0 + m_1}^1) }{\det( df^{k - (m_0 + m_1)}_{\uo^+} | T \gamma_{m_0 + m_1}) ( p_{m_0 + m_1}^1) } \bigg| 
\end{gather*}
The first RHS term is the sum of $m_0 + m_1$ terms, each of which is bounded from above in terms of $\| df_{\omega_j}\|, \| df_{\omega_j}^{-1}\|$, $1 \leq j \leq m_0 + m_1$; these in turn are controlled by the value $l_1(\hat \uo) \leq \hat l^{(0)}_{(x, \uo^+)}$ of the function $l_1$ as in Section 2.1, (recall $\hat \uo = \hat \uo(x, \uo^+)$ as in Section 3.1).
By these considerations, this term is bounded $\leq D_1$, where $D_1 = D_1(K_0, a_0,  r_0; \hat l^{(0)})$ (noting that $m_0, m_1$ depend on $K_0, a_0,  r_0$). 

For the second RHS term, observe that the graphing functions $g_j, j \geq m_0 + m_1$, satisfy $\Lip(g_j) \leq 1/10$ and $\Lip(d g_j) \leq \bar C \hat l^{(j)}$. A distortion estimate analogous to that in Lemma \ref{lem:distortionHatSc}
applies to bound this term $\leq D_2(\hat l^{(k)}) \cdot d(p_1, p_2)$. 

The proof is complete on setting $D = D_1 + D_2$.
\end{proof}

{The next lemma gives sufficient conditions for switching of axes (Lemma \ref{lem:chartSwitch11}) in the present context.} Let 
$\check g_k : \operatorname{Dom}(\check g_k) \subset \hat E^{u, (k)} \to E^{cs, (k)}$ 
be given by $\check g_k = \hat L^{(k)}_{(x, \uo^+)} \circ g_k \circ 
(\hat L^{(k)}_{(x, \uo^+)})^{-1}$. 

%
%

\begin{lem}\label{lem:chartSwitchWn}
For any $\hat l > 1$ there exists $r_3 = r_3(\hat l)$ with the following properties. Let $k \geq m_0 + m_1$ and let $r_3 = r_3(\hat l^{(k)})$. Then
\begin{itemize}
\item[(a)] $\operatorname{Dom}(\check g_k)$ contains $\hat E^{u, (k)}(r_3)$; and
\item[(b)] if $|(d g_k)_0| \leq (20 \check l^{(k)})^{-1}$ holds, then we have $\calLip(\check g_k) \leq 1/10$ on $\hat E^{u, (k)}(r_3)$. 
\end{itemize}
\end{lem}

Since $|(d g_k)_0| \leq K_0 e^{- k \lambda / 2}$ (Proposition \ref{lem:inclination}) and {$\hat l^{(k)} \leq e^{k \d_2} \hat l^{(0)}$} (Proposition \ref{prop:charts2Side}), the condition in (b) is satisfied for all large enough $k$ depending on $\hat l^{(0)}$ and $K_0$.

\begin{proof}
Item (a) is guaranteed when $r_3(\hat l)$ is taken $\leq (\bar r_1 \hat l^2)^{-1}$. For (b), for $r > 0$ we estimate $\calLip(\check g_k|_{\hat E^{u, (k)}(r)})$ as follows. Let $\check u \in \hat E^{u, (k)}(r)$, $\check u = \hat L^{(k)}_{(x, \uo^+)} u$, $u \in \R^u$, and estimate
\begin{align*}
\| (d \check g_k)_{\check u} \| & \leq \| (d \check g_k)_0 \| + \calLip(d \check g_k) \| \check u \| 
 \leq \hat l^{(k)} | (d g_k)_0| + (\hat l^{(k)})^2 \Lip(d g_k) \| \check u \| \\
 & \leq \hat l^{(k)} | (d g_k)_0| + \bar C  (\hat l^{(k)})^3 \cdot r \, ,
\end{align*}
where $\bar C$ is as in the end of Section 3.2. Taking $r_3(\hat l) \leq (20 \bar C \hat l^3)^{-1}$ ensures the second RHS term is $\leq 1/20$, while the first term is $\leq 1/20$ when the condition in (b) is met.
\end{proof}

\medskip

\noindent {\it In the rest of the paper, $\bar K_0$ is fixed, as is $\d \in (0, \bar r_1(\bar K_0))$ 
sufficiently small for the purposes of Proposition \ref{prop:localUMfld} and Lemmas \ref{lem:graphTransform}, \ref{lem:distortionHatSc}.}


\section{Setup for the rest of the proof}

For $\uo \in \Omega^\Z$, we will realize $\mu_\uo$ as the weak limit as $n \to\infty$
of $\mu^n_\uo = (f^n_{\theta^{-n}\uo})_* \mu$, obtained by pushing $\mu$ forward
from time $-n$ to $0$. But we will not push forward all of $\mu$, only a small bit of it,
as that is all that is needed to show $\mu^*$ has the SRB property; see Sect. 2.5.
As a matter of fact, we will push forward a very localized bit of $\mu$ ({located
on a} ``source set''), and consider only the part that arrives in a localized region 
(the ``target set''), both suitably chosen. {This section describes 
and justifies the main ingredients of this setup; details including order of choice of constants are given in Sect. 5.1.}

\subsection{Uniformity sets of $\mu^n_\uo$-typical points}\label{subsec:112unifSets}

Let $l_0 > 1$ be fixed implicitly throughout. We fix a compact subset 
$\Theta_0 \subset \Gamma_{l_0}$, and for now fix $n \in \Z^+$.
We define $\Theta_n := \Theta_0 \cap  \tau^{-n} \Theta_0$, so that $\Theta_n$
consists of points $(x, \uo)$ that are good in the sense of being in a uniformity set
both at time $0$ and at time $n$. As in Sect. 3.1, a measurable selection
(Lemma \ref{lem:measPastChoice}) 
$\hat \uo^n : \Theta_n^+ \to \Omega^{\Z^-}$ enables us to systematically
assign ``pasts" to points in $ \Theta_n^+$, a positive $(\mu \times \P^+)$-measure set.

Now since we are interested in $\mu^n_\uo = (f^n_{\theta^{-n}\uo})_* \mu$,
we want to consider orbits starting from time
$-n$ and not from $0$. For $\uo \in \Omega^\Z$, we write $x_{-n}= f^{-n}_{\uo} x$, and 
define
\[
M^n_\uo = \{ x \in M : (x_{-n}, (\theta^{-n} \uo)^+) \in \Theta_n^+ \} \, .
\]
Then $M^n_\uo$ is a subset of $\mu^n_\uo$-typical points.

The ideas from Sect. 3.1 carry over in a straightforward way, though the notation gets
 more cumbersome. 
Let $x \in M^n_\uo$. For $k=0,1,\dots, n$, let $x_{-k} = f^{-k}_{\uo} x$. 
As before, 
$$
E^{cs}_{(x_{-k}, (\theta^{-k} \uo)^+)} = E^{cs}_{(x_{-k}, \theta^{-k} \uo)}
$$ 
are well defined, as $E^{cs}$-subspaces depend only on the future. 
To define $E^u$, for brevity let us write 
$\hat \uo=(\hat \uo^n(x_{-n}, (\theta^{-n}\uo)^+),(\theta^{-n}\uo)^+)$.
We define 
$$
\hat E^{u, (0)}_{(x_{-n}, (\theta^{-n} \uo)^+)} = \hat E^u_{(x_{-n}, (\theta^{-n} \uo)^+)} = E^u_{(x_{-n}, \hat \uo)}\ ,
$$
and for $k=0, 1, \dots, n-1$, we define the $E^u$-subspace at 
$(x_{-k}, (\theta^{-k} \uo)^+)$ as
$$
\hat E^{u, (n-k)}_{(x_{-n}, (\theta^{-n} \uo)^+)} = 
E^u_{\tau^{n-k}(x_{-n}, \hat \uo)}\ .
$$
Chart maps $\hat \Phi^{(n-k)}_{(x_{-n}, (\theta^{-n} \uo)^+)}$, connecting maps
$\hat f^{(n-k)}_{(x_{-n}, (\theta^{-n} \uo)^+)}$ and the $l$-function
$\hat l^{(n-k)}_{(x_{-n}, (\theta^{-n} \uo)^+)}$ are all defined as before.
Observe that by our choice of $\Theta_n$, we have that 
$$
\hat l^{(0)}_{(x_{-n}, (\theta^{-n} \uo)^+)}\ , \hat l^{(n)}_{(x_{-n}, (\theta^{-n} \uo)^+)} \le l_0\ .
$$

We finish by recording the following observation. Let $M_\uo  =\{ x \in M : (x, \uo) \in \Theta_0\}$; that is, $M_\uo$ is a uniformity set for the two-sided dynamics restricted to the fiber $M \times \{ \uo \}$. Since $\mu^n_\uo \to \mu_\uo$ weakly, one should expect that as $n \to \infty$, the uniformity set $M^n_\uo$ of $\mu^n_\uo$-typical points should converge to $M_\uo$ in some sense. Below this is made precise.
\begin{lem}\label{lem:convergeUnifSet}
For any $\d > 0$, there exists $N_0 = N_0(\d) \in \N$ such that for any $n \geq N_0$, we have that $M_{\uo}^n \subset \Nc_\d(M_{\uo})$. In particular,
\[
\lim_{n \to \infty} \sup_{x \in M^n_\uo} \dist(x, M_\uo) = 0 \, .
\]
\end{lem}
\begin{proof}
By standard compactness arguments, it suffices to prove that for any sequence $\{ x^n \} \subset M$ converging to a point $x \in M$ for which $x^n \in M^n_\uo$ for all $n$, we have that $x \in M_\uo$. For each $n \geq 1$ write $x_{-n}^n = f^{-n}_\uo x^n$, and
let $\check \uo_n \in \Omega^\Z$ be defined by $\check \uo_n = \theta^n (\hat \uo^n(x^n_{-n}, (\theta^{-n} \uo)^+), (\theta^{-n} \uo)^+)$. Observe that $\check \uo_n \to \uo$ as $n \to \infty$, and that $(x^n, \check \uo_n) \in \tau^{n} \Theta_n \subset \Theta_0$ for all $n \geq 0$ by our measurable selection construction. Since
 $(x^n, \check \uo_n)$ converges to $(x, \uo)$, and $\Theta_0$ is compact, we
obtain that $(x,\uo) \in \Theta_0$, i.e., $x \in M_\uo$.
\end{proof}

\newcommand{\ur}{{\underline r}}
\newcommand{\ck}{\mathfrak{c}}

\subsection{Accumulating $\mu^n_\uo$-mass}

Let $\beta_0>0$ be a very small number. We fix $l_0 > 1$ sufficiently large so that $\mu^* \Gamma_{l_0} \geq 1- \beta_0/3$. Fix a compact set $\Theta_0 \subset \Gamma_{l_0}$ 
with $\mu^* \Theta_0 \geq 1 - \beta_0/2$ and for each $n \geq 0$ let $\Theta_n = \Theta_0 \cap \tau^{-n} \Theta_0$, so that $\mu^*(\Theta_n) \geq 1 - \beta_0$.
For each $n$, fix a measurable selection $\hat \uo^n : \Theta_n^+ \to \Omega^{\Z^-}$ as in Lemma \ref{lem:measPastChoice}.
Finally, $M^n_\uo$ is as defined in Sect. 4.1.

Below, we determine a set of $\uo$ for which $M^n_\uo$ has sufficiently large $\mu^n_\uo$-mass for an infinite sequence of $n$.

\begin{lem}\label{lem:largeEnoughUnifSet}
Let $\ck  > 1$. For each $n \geq 1$ define
\[
\Gc^{(n)} = \{ \uo \in \Omega^\Z : \mu^n_\uo( M^n_\uo ) \geq 1 - \ck \cdot \beta_0\} \, ,
\]
and set $\Gc = \limsup_{n \to \infty} \Gc^{(n)} = \cap_{N \geq 1} \cup_{n \geq N} \Gc^{(n)}$. Then, we have
\[
\P(\Gc) \geq \frac{\ck - 1}{\ck} \, .
\]
\end{lem}

\begin{proof}

We claim that
\begin{align}\label{eq:GcPullBackIdentity}
\theta^{-n} \Gc^{(n)} = \big\{ \uo \in \Omega^\Z : \mu\{ y \in M : (y, \uo^+) \in \Theta^+_n \} \geq 1 - \mathfrak c \b_0 \big\}
\end{align}
Assuming this for the moment, observe that $\theta^{-n} \Gc^{(n)}$ depends only on the $\Omega^{\Z^+}$-coordinate of $\uo$, hence
\[
\P^+((\theta^{-n} \Gc^{(n)})^+) =  \P( \theta^{-n} \Gc^{(n)}) = \P( \Gc^{(n)}) 
\]
by the $\P$-invariance of the shift $\theta$.
We now estimate:
\begin{align*}
1 -  \beta_0 & \leq  (\mu \times \P^+)( \Theta^+_n) \\
& = \bigg( \int_{(\theta^{-n} \Gc^{(n)})^+} +  \int_{\Omega^{\Z^+} \setminus (\theta^{-n} \Gc^{(n)})^+} \bigg) \mu \{ x \in M : (x, \uo^+) \in \Theta^+_n\} \, d \P^+( \uo^+) \\
& \leq \P^+( (\theta^{-n} \Gc^{(n)})^+) + (1 - \ck \beta_0) (1 - \P^+((\theta^{-n} \Gc^{(n)})^+))
\end{align*}
Rearranging, we obtain $\frac{\ck - 1}{\ck} \leq \P^+((\theta^{-n} \Gc^{(n)})^+) = \P(\Gc^{(n)})$, hence $\P( \Gc) \geq \frac{\ck - 1}{\ck}$. 

It remains to check \eqref{eq:GcPullBackIdentity}. Observe that $\theta^n \uo \in \Gc^{(n)}$ 
holds iff $\mu^n_{\theta^n \uo} (M^n_{\theta^{n} \uo}) \geq 1- \mathfrak c \b_0$. We evaluate
\begin{align*}
M^n_{\theta^n \uo} & = \{ x \in M : ( f^{-n}_{\theta^n \uo} x, \uo^+) \in \Theta^+_n \}
 = \{ x \in M : ((f^n_{\uo})^{-1} x, \uo^+) \in \Theta^+_n \} \\
 & = f^n_\uo \{ y \in M : (y, \uo^+) \in \Theta^+_n \} \, .
\end{align*}
Since $\mu_{\theta^n \uo} = (f^n_\uo)_* \mu$, equation \eqref{eq:GcPullBackIdentity} follows immediately.
\end{proof}



\bigskip
Our next step is to coordinate for each $\uo \in \Gc$ for a positive amount of
$\mu$-mass to come from a small, fixed region (the ``source set'') and to land in a small, 
fixed region (the ``target set'') under $f^{n_i}_{\theta^{-n_i}\uo}$ for some infinite sequence $\{n_i\}$. 

\medskip

We write $\psi := \frac{d \mu}{d \Leb}$, which we recall is continuous by hypothesis. With $\b_0$ as before, let us define $\a_0 = \b_0 / \Leb(M)$, so that
\[
\mu\{ \psi \geq \a_0\} \geq 1 - \b_0 \, .
\]

\begin{lem}\label{lem:accumulateMass}
For any $\e > 0$, there exists a constant $c = c(\e) > 0$ such that for any $\uo \in \Gc$, we have the following. There are points $\hat p_- \in \{ \psi \geq \a_0\}, \hat p \in M$, and a sequence $n_i \to \infty$ for which
\begin{gather}
\mu^{n}_{\uo} \big( M^{n}_\uo \cap \Bc(\hat p, \e) \cap f^{n}_{\theta^{-{n}} \uo} \Bc(\hat p_{-}, \e)  \big) \geq c \qquad \mbox{for all } n = n_i \ .
\end{gather}
\end{lem}

Note that in Lemma \ref{lem:accumulateMass}, the points 
$\hat p, \hat p_- \in M$ and the subsequence $n_i$ all depend on $\uo$,
whereas the constant $c = c(\e)$ is independent of $\uo$.
\begin{proof}
Let $\uo \in \Gc$. To start, fix a subsequence $n_i \to \infty$ along which $\uo \in \Gc^{(n)}$ for all $n = n_i$.
In pursuit of the `source set' $\Bc(\hat p_-, \e)$ and `target set' $\Bc(\hat p, \e)$, we refine $(n_i)$ successively
several times in the following argument.

Fix an open cover of $\{ \psi \geq \a_0\}$ by balls of radius $\e$ with centers $p_j \in  \{ \psi \geq \a_0\}, 1 \leq j \leq J$. For each $n = n_i$, 
we estimate:
\[
c_- := \frac{1}{J} (1- (\ck+1) \b_0) \leq \frac{1}{J} \mu ( f^{-n}_\uo M^{n}_\uo \cap  \{ \psi \geq \a_0\}) 
\leq \frac1J \sum_{j = 1}^J \mu (\Bc(p_j, \e) \cap f^{-n}_\uo  M^{n}_\uo) \, .
\]
For each $i$, there exists $j = j(i)$ so that $\mu(\Bc(p_j, \e) \cap f^{-n_i}_\uo M^{n_i}_\uo) \geq c_-$. 
Since there are only finitely many $j$, by the Pidgeonhole principle we may refine $(n_i)$ so that
\[
\mu^{n_i}_{\uo}(\Bc(\hat p_-, \e)) \cap M^{n_i}_\uo) \geq c_-
\] 
holds for $\hat p_- = p_{\hat j_-}$ for some fixed $\hat j_- \in \{1, \cdots, J\}$.

Continuing, fix an open cover of $M$ by balls of radius $\e$ with centers $p_j' \in M, 1 \leq j \leq J'$. For each $n = n_i$, we estimate:
\begin{align*}
c :=  \frac{c_-}{J'} & \leq \frac{1}{J'} \mu^{n}_{\uo}( M^{n}_\uo \cap f^{n}_{\theta^{-n} \uo} \Bc(\hat p_-, \e))   \leq \frac1{J'} \sum_{j = 1}^{J'} \mu^{n}_{\uo} \big(\Bc(p_j', \e) \cap M^{n}_\uo \cap f^{n}_{\theta^{-n} \uo} \Bc(\hat p_-, \e) \big) 
\end{align*}
By the same Pidgeonhole Principle argument, on refining $(n_i)$ once more we have that 
\[
\mu_{\uo}^{n_i} \big( \Bc(\hat p, \e) \cap M^{n_i}_\uo \cap f^{n_i}_{\theta^{-n_i} \uo} \Bc(\hat p_-, \e) \big) \geq c
\] 
where $\hat p = p_{\hat j}'$ for some fixed $\hat j \in \{ 1, \cdots, J'\}$. 
\end{proof}

\subsection{Disintegration of $\mu$ in the ``source set" onto $u$-graphs}

Fix $\uo \in \Gc$. We assume in this subsection that $\e > 0$ is specified, 
and Lemma \ref{lem:accumulateMass} has been applied to obtain $\hat p_-, \hat p \in M$, a 
sequence $\{n_i\}$ and $c = c(\e) > 0$. For $n=n_i$, define
\begin{align}\label{eq:defineLambda}
\Lambda^n = M_{\uo}^n \cap \overline{\Bc(\hat p, \e)} \cap f^{n}_{\theta^{-n} \uo} \overline{ \Bc(\hat p_-, \e)}
\qquad \mbox{and} \qquad \Lambda^n_- = f^{-n}_{\uo}(\Lambda^n)\ .
\end{align}
We now specify how $\mu$ restricted to $\Lambda^n_-$ will be decomposed
into measures on graphs to be pushed forward.

For $x \in \Lambda^n_-$, we have $(x, (\theta^{-n}\uo)^+) \in \Theta_n^+$.
This means in particular that $(x, (\theta^{-n}\uo)^+)$ possesses
a natural $E^{cs}_{(x, (\theta^{-n}\uo)^+)}$ and a systematic and measurable assignment of 
$\hat E^{u,(0)}_{(x, (\theta^{-n}\uo)^+)}$.
We will distintegrate $\mu$ onto the leaves of 
a smooth foliation $\mathcal F^n_-$, chosen in such a way that the leaves of 
$\mathcal F^n_-$, suitably restricted,
 are graphs from open subsets of $\hat E^{u,(0)}_{(x, (\theta^{-n}\uo)^+)}$ to 
$E^{cs}_{(x, (\theta^{-n}\uo)^+)}$ for each $x \in \Lambda^n_-$; we will refer to
such graphs as ``$u$-graphs".

\smallskip
To define $\mathcal F^n_-$, we fix a reference point $q^n_- \in \Lambda^n_-$.
As the discussion is entirely local, we confuse a neighborhood of $q^n_-$ in $M$
with a subset of $T_{q^n_-}M$  via $\exp_{q^n_-}$ and
define $\mathcal F^n_-$ to be the collection of 
$(\dim E^u)$-dimensional hyperplanes in $T_{q^n_- }M$ parallel to 
$\hat E^{u,(0)}_{(q^n_-, (\theta^{-n}\uo)^+)}$.

\begin{lem}\label{lem:lamination}
For all sufficiently small $\e > 0$ depending on $l_0, \psi = \frac{d \mu}{d \Leb}$ and 
$\a_0$, there exist constants $K_- = K_-(l_0), r_- = r_-(l_0, \e)$ for which the following hold for all $\uo \in \Gc$.
Assume Lemma \ref{lem:accumulateMass} has been applied. Let $n=n_i$, and let
 $\Fc^n_-$ be as above. Then for every $x \in \Lambda^n_-$, 
\begin{itemize}
\item[(a)] there is a function $h_x^- = h^-_{(x, (\theta^{-n} \uo)^+)}$\,,
$$
h_x^-:  B^u(r_-(\hat l^{(0)}_{(x, (\theta^{-n}\uo)^+)})^{-1}) \to \R^{cs} \qquad \mbox{with} \qquad
\Lip(h_x^-) \ ,  \Lip(dh_x^-) < K_- \, ,
$$
such that if $\mathcal F^n_-(x)$ is the leaf of $\mathcal F^n_-$ containing $x$, then
$$
\hat \Phi^{(0)}_{(x, (\theta^{-n}\uo)^+)}(\graph h_x^-) \subset \mathcal F^n_-(x)\ ;
$$
\item[(b)] $\hat \Phi^{(0)}_{(x, (\theta^{-n}\uo)^+)}  B(r_-(\hat l^{(0)}_{(x, (\theta^{-n}\uo)^+)})^{-1})
 \subset \{ \psi \geq \alpha_0 / 2\}$.
\end{itemize}
\end{lem}

\begin{proof} First we require $\e$ to be small enough that $\Nc_{2 \e} \{ \psi \ge \a_0 \} \subset \{ \psi \ge \a_0 / 2\}$, where $\Nc_{ \e}$ denotes the $\e$-neighborhood of a set (recall that the density $\psi$ of $\mu$ was assumed continuous-- see Section 2). It follows that $\Bc(\hat p_-, 2 \e) \subset
\{\psi \ge \alpha_0/2\}$.

To check (a) and (b) for $n=n_i$, observe that by Lemma \ref{prop:ctyEucs}, 
$x \mapsto E^{cs}_{(x, (\theta^{-n}\uo)^+)}$ varies continuously
for $(x, (\theta^{-n}\uo)^+) \in \Gamma^+_{l_0}$, and by Remark \ref{rmk:ctyEu}(a) we have
$\|\hat \pi^{u, (0)}_{(x, (\theta^{-n}\uo)^+)}\| \le l_0$. This means that by choosing 
$\e$ sufficiently small to align nearby $E^{cs}$-subspaces, we are assured that there is uniform separation between 
$\hat E^{u,(0)}_{(q^n_-, (\theta^{-n}\uo)^+)}$,
i.e., $\mathcal F^n_-(x)$, 
and $E^{cs}_{(x, (\theta^{-n}\uo)^+)}$ for all $x \in \Lambda^n_-$.
This separation provides a constant $K_-$ in (a). That is, if $h_x^-$ is chosen
to satisfy
$\hat \Phi^{(0)}_{(x, (\theta^{-n}\uo)^+)}(\graph h_x^-) \subset \mathcal F^n_-(x)$,
then $\Lip(h_x^-), \Lip(d h_x^-) < K_-$.
We shrink $r_-$ as needed to ensure that $\hat \Phi^{(0)}_{(x, (\theta^{-n} \uo)^+)} \graph h_x^- \subset \Bc(\hat p_-, 2 \e)$, hence 
(b) holds.

Now the constants above need to work for all $n$, and not be chosen
one $n=n_i$ at a time. This requires that
the modulus of continuity of $x \mapsto E^{cs}_{(x, (\theta^{-n}\uo)^+)}$ 
be independent of $(\theta^{-n} \uo)^+$, which is true because
$(x, (\theta^{-n} \uo)^+)$ is contained in the compact set $ \Theta^+_n \subset \Theta^+_0$.
\end{proof}

\medskip
Returning to the measure $\mu$, and continuing to confuse a neighborhood of 
$M$ with a subset of $T_{q^n_-}M$,
we have that $(\frac{\alpha_0}{2}{\rm Leb}|_{B_{2\e}(\hat p_-)})\le \mu$, and
the conditional densities of $\frac{\alpha_0}{2}{\rm Leb}|_{B_{2\e}(\hat p_-)}$  
on the leaves of $\mathcal F^n_-$ are constant functions.


\begin{figure}[ht]
  \centering

\includegraphics[width=\textwidth]{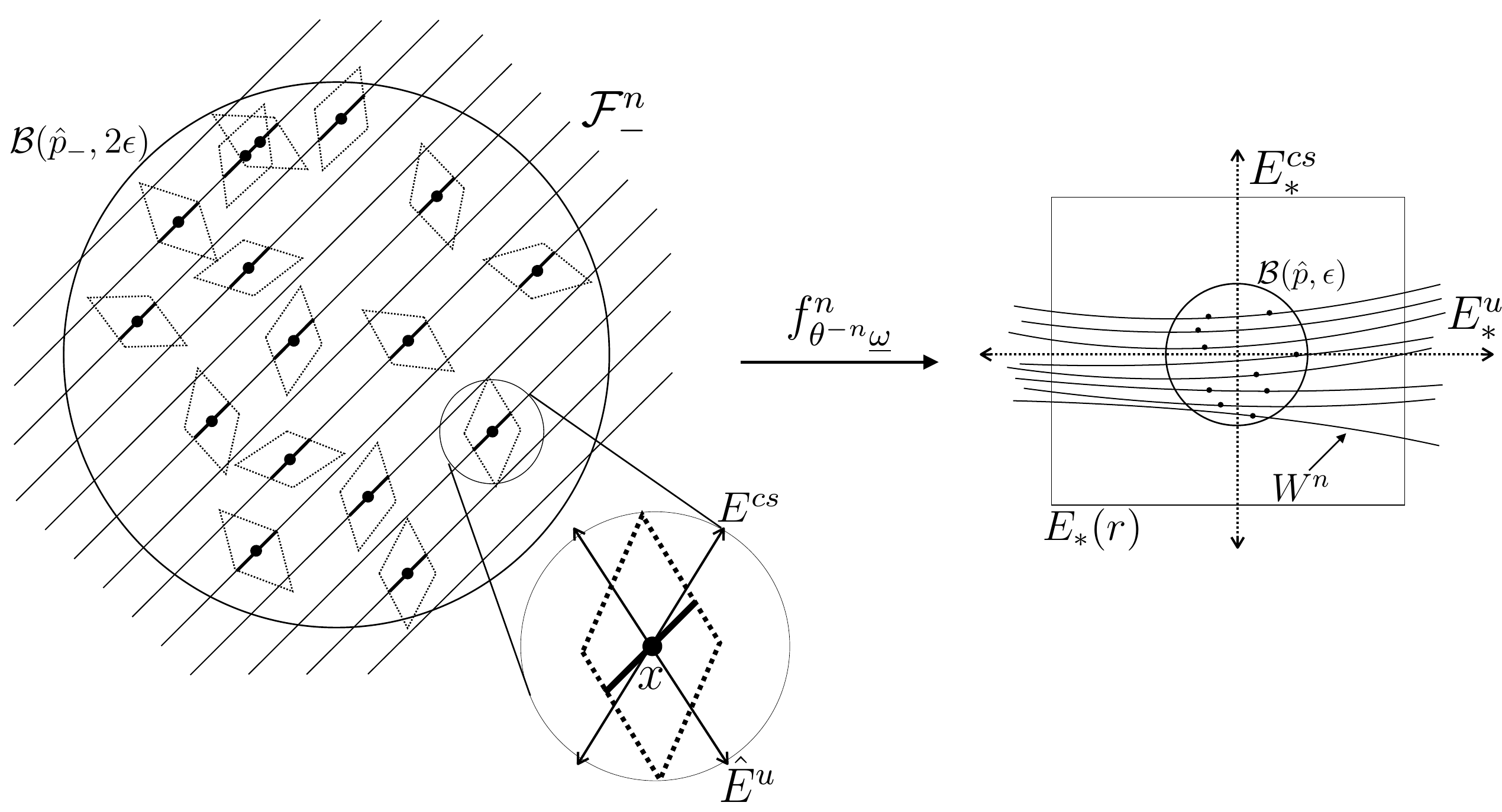}

  \caption{ 
  On the left-hand side is the foliation $\Fc^n_-$ through the ball $\Bc(\hat p_-, 2 \e)$ (the ``source set''). 
Surrounding each $x \in \Lambda_-^n$ is the `chart box' $\hat \Phi^{(0)}_{(x, (\theta^{-n} \uo)^+)} B(r_- (\hat l^{(0)})^{-1})$ (parallelogram); the bolded portion of the $\Fc^n_-$-leaf through each $x$ is a $u$-graph represented by the graphing function $h_x^-$ as in Lemma \ref{lem:lamination}. Note that the chart boxes at $x \in \Lambda^n_-$ are \emph{not} well-aligned with each other, and the functions $h_x^-$ may have large slopes. On the right side (the ``target set'') is a rectangle of the form
 $E^u_*(r)\times E^{cs}_*(r)$ centered at a reference point 
 $x_* \in \mathcal B(\hat p,\epsilon)$. 
For $n$ large enough, $f^n_{\theta^{-n} \uo}$ maps portions of the little pieces of 
$u$-graphs on the left across this box, intersecting it in 
sets that are graphs of functions from $E^u_*(r)$ to $E^{cs}_*(r)$ as shown.}
\end{figure}

\section{Geometry of pushed-forward $u$-graphs}

We have set up in Section 4 a situation that can be described as follows:
For each $\uo$ in a positive $\P$-measure set, there is a sequence $n_i$ 
such that for each $n=n_i$, there is a set $\Lambda^n_- \subset M$, and 
a collection of $u$-graphs associated with $x \in \Lambda^n_-$ that together 
carry positive $\mu$-measure. These $u$-graphs are to be transported 
forward by $f^n_{\theta^{-n}\uo}$. {We consider small 
pieces of these $u$-graphs that remain inside suitable Lyapunov charts 
for all $n$ steps, and refer to
the images at the end as $W^n$-leaves.}
In this section, we will focus on the {\it geometry} of the 
$W^n$ leaves and the manner to which they converge to (real) unstable manifolds of the RDS.
 We will begin by making precise the order of the various choices of constants 
and constructions.


\subsection{Stacks of $W^n$-leaves: details of construction}

We now bring together the following
three sets of ingredients we have prepared: the setup in Section 4 in which we accumulate certain sets 
of points with controlled finite pasts (Lemmas \ref{lem:accumulateMass} and 
\ref{lem:lamination}), graph transforms for `slanted' graphs developed 
in Sects. 3.2 -- 3.3 (Proposition \ref{lem:inclination}), and the switching of axes and consolidation
of images onto stacks (Lemma \ref{lem:chartSwitch11}).

\bigskip

\noindent {\it (A) Initial choices. } 
 To start, fix a small $\b_0 > 0$ and let $l_0 > 1$ be such that $\mu^*\{ l \leq l_0\} \geq 1 - \b_0 / 3$. 
With $\Theta_0, \Theta_n \subset \Gamma$ as in the beginning of Section 4.2, let $\Gc$
be as in Lemma \ref{lem:largeEnoughUnifSet} with $\mathfrak c = 2$ so that $\P(\Gc) \geq 1/2$.
In all that follows, $\uo \in \Gc$ is fixed. As previously,
we write $M_\uo = \{ x \in M : (x, \uo) \in \Theta_0\}$ and, for $n \geq 0$, $M_\uo^n = 
\{ x \in M : (x_{-n}, (\theta^{-n} \uo)^+) \in \Theta_n^+\}$.

\bigskip

\noindent {\it (B) Choices of $\e_*, r_*$, source and target sets, and a lower bound for $\{n_i\}$. }
Our aim by the end of part (B) is to have constructed the following objects:
\begin{itemize}
\item[(a)] `source' and `target' sets $\Bc(\hat p_-, \e_*), \Bc(\hat p, \e_*)$ (as in Lemma \ref{lem:accumulateMass});
\item[(b)] a reference point $x_* \in M_\uo \cap \Bc(\hat p, \e_*)$ and a reference box 
$E_*(r_*) := E^u_*(r_*) \times E^{cs}_*(r_*)$, $E^{u/cs}_* := E^{u/cs}_{(x, \uo)}$, 
suitable for constructing stacks of (i) $W^u$-leaves (as in Lemma \ref{lem:uStack2Side}) 
and (ii) appropriately truncated, pushed-forward $u$-graphs (called $W^n$-leaves) through 
the target set $\Bc(\hat p, \e_*)$ (see Figure 1).
\end{itemize}
The main work in constructing (a) and (b) is to identify the parameters $\e_*, r_*$, which we undertake now, starting with $r_*$.  

For (b)(i), Lemma \ref{lem:uStack2Side} requires that we take $r_*$ sufficiently small 
in terms of $l_0$, $\d$. For (b)(ii), to each $x \in \Lambda^n$
is associated a graph-transform-image (in the sense of Section 3) of a u-graph at $x_{-n} := f^{-n}_\uo x$ (to be made 
precise in (C) below).  The image, what we call a $W^n$-leaf, will be a graph defined on 
the $(\hat E^{u, (n)}_{(x_{-n}, (\theta^{-n} \uo)^+)}, E^{cs, (n)}_{(x_{-n}, (\theta^{-n} \uo)^+)})$-axes 
in $T_x M$. We seek to
switch axes to a common reference box $E_*(r_*) = E^{u}_*(r_*) \times E^{cs}_*(r_*)$
centered at a reference point $x_*$ (to be determined).
Taking $r_* \leq \frac12 r_3$ with $r_3 = r_3(l_0)$ as in Lemma \ref{lem:chartSwitchWn}
 ensures that truncations of $W^n$ leaves will have small-enough $\calLip$ constants
 for the purposes of Lemma \ref{lem:chartSwitch11}(iii),
provided that the $W^n$-leaves are sufficiently parallel to 
$\hat E^{u, (n)}_{(x_{-n}(\theta^{-n} \uo)^+)}$ (see end of (C)). 
This completely fixes the value of $r_*$.

We now identify two sets of conditions on $\e_*$, to be used in the construction of
 the `source' and `target' sets.
At the source set, Lemma \ref{lem:lamination} imposes two conditions on $\e_*$: 
One is that it has to be small enough so that $E^{cs}$ is sufficiently well-aligned through points of $f^{-n}_\uo M_\uo^n$
when restricted to a ball of radius $\e_*$; 
this is needed to guarantee the separation of 
$E^{cs}$ and $E^u$ in the sense of Remark \ref{rmk:ctyEu}(a). The other is that the entire $2\e_*$-ball
should be contained in $\{\psi \ge \alpha_0/2\}$ (Lemma \ref{lem:lamination}(b)).

At the target set we require $\e_*$ be suitable for constructing the stacks of both $W^u$ 
and $W^n$ leaves. Both require that $\e_*$ be small enough in terms of $l_0, \d$ and $r_*$ 
(Lemmas \ref{lem:chartSwitch11} and \ref{lem:uStack2Side}). Additionally, for the $W^n$ 
stack we need to make the $\hat E^{u, (n)}, E^{cs}$-axes at $x \in \Lambda^n$ line up with 
the $E^{u/cs}_*$ axes at the reference point $x_*$. The $E^{cs}$ axes are aligned by 
shrinking $\e_*$ (Proposition \ref{prop:ctyEucs}(a)); to align
the $\hat E^u$ with $E^u_*$ requires shrinking $\e_*$ and taking $\min \{ n_i\}$ sufficiently 
large (Proposition \ref{prop:ctyEucs}(b) and Remark \ref{rmk:ctyEu}(b)).

These are our requirements on $\e_*$ and $r_*$. With $\e_*$ determined, we are correctly
situated to apply Lemma \ref{lem:accumulateMass} with $\e = \e_*$, fixing once and for all
 the `source' and `target' regions $\Bc(\hat p_-, \e_*), \Bc(\hat p, \e_*)$ respectively, 
 and the {potentially viable
subsequence ${n_i}$} along which we have the bound $\mu^n_\uo(\Lambda^n) = \mu(\Lambda^n_-) \geq c_*  $ for $n = n_i$; here $c_* := c(\e_*)$ is as in Lemma \ref{lem:accumulateMass} and $\Lambda^n_-, \Lambda^n$ are the `source' and `target' sets as in \eqref{eq:defineLambda}. Finally, we fix 
an arbitrary point $x_* \in M_\uo \cap \Bc(\hat p, \e_*)$ to be used as reference point; 
that $M_\uo \cap \Bc(\hat p, \e_*) \neq \emptyset$ is guaranteed by 
Lemma \ref{lem:accumulateMass}. This completes  the construction of $E_*(r_*)$.

Further conditions will be imposed on the lower bound for $\{n_i\}$. 

\bigskip
\noindent {\it (C) Pushing forward $\Fc^n_-$-leaves. } For each $x \in \Lambda^n$,
we let $x_{-n} = f^{-n}_\uo x \in \Lambda^n_-$, and push forward 
 the graph of $h_{x_{-n}}^-$ (here $x_{-n}$ plays the role of $x$ in Lemma \ref{lem:lamination}) by applying Proposition \ref{lem:inclination}. 
Letting $K_0 = K_- = K_-(l_0)$ and $r_- = r_-(l_0, \e_*)$  be as in 
Lemma \ref{lem:lamination}, and $\bar r_0 = \bar r_0(K_-)$ be as in Proposition \ref{lem:inclination}, we set $r_0 = \min\{ \bar r_0, r_-\}$.
Then for all $n = n_i \geq m_0 + m_1$ where 
 $m_0, m_1$ are as in Proposition \ref{lem:inclination} (depending on $r_-, K_-$),
the graph transform $h_x := h^n_{(x, \uo)} = \mathcal T^{(n-1)} \circ \cdots \circ \mathcal T^{(0)} 
h^-_{x_{-n}}$ is defined on  all of $B^u( \bar r_1 (\hat l^{(n)}_{(x_{-n}, (\theta^{-n} \uo)^+)})^{-1})$ 
and satisfies
 \begin{gather}\label{eq:controlgnGraph}\begin{gathered}
h_x(0) = 0, \quad \Lip(h_x) \leq \bar K_0 \leq 1/10 \, , \\
| (d h_x)_0| \leq K_- e^{- n {\lambda /2}} \, , \quad \text{ and } \quad \Lip(d h_x) \leq \bar C \hat l^{(n)}_{(x_{-n}, (\theta^{-n} \uo)^+)} \, ,
\end{gathered}\end{gather}
where $\bar C$ is as in Proposition \ref{lem:inclination}. We define
\[
W^n_{(x, \uo)} := \hat \Phi^{(n)}_{(x_{-n}, (\theta^{-n} \uo)^+)} \graph h_x 
\]
and let $\check h_x : \operatorname{Dom}(\check h_x) \subset \hat E^{u, (n)}_{(x_{-n}, (\theta^{-n} \uo)^+)} \to E^{cs, (n)}_{(x_{-n}, (\theta^{-n} \uo)^+)}$ be its graphing map. To perform the switching of axes 
to $E^{u/cs}_*$ as discussed in Paragraph (B), we guarantee
$\| (d \check h_x)_0 \| \leq \frac{1}{20}$ (Lemma \ref{lem:chartSwitchWn})
by taking $\min \{ n_i \}$ sufficiently large
so that $| (d h_x)_0| \leq (20 l_0)^{-1}$.

\bigskip
\noindent {\it (D) Collecting onto stacks.} Let 
\[
\Lambda := \overline{\Bc(\hat p, \e_*)} \cap M_\uo  
\]
where $\hat p \in M$ is as in (B). 
Applying Lemma \ref{lem:uStack2Side}, we let $\Sc = \cup_{x \in \Lambda} \xi(x)$ 
be the stack of $W^u$-leaves and $\Xi$ the partition of $\Sc$ into $\xi(x)$, where
$\xi(x)$, the $W^u$ leaf through $x \in \Lambda$, has the form $\xi(x) = \exp_{x_*} \graph \gamma_x$ {for some}
 $\gamma_x : E_*^u(r_*) \to E_*^{cs}(r_*)$ with $\calLip(\gamma_x) \leq 1$.

For $n = n_i$ sufficiently large, we now define the corresponding stack $\Sc^n$:
\begin{lem}\label{lem:construct112stack}
There exists $N_* \in \N$, depending on all the parameters above,  
such that the following holds for all $n = n_i \geq N_*$.
\begin{itemize}
\item[(a)] 
For each $x \in \Lambda^n$, we have that the connected component of $W^n_{(x, \uo)} \cap \exp_{x_*} E_*(r_*)$ 
containing $x$ coincides with $\exp_{x_*} \graph \gamma^n_x$,
where $\gamma^n_x : E_*^u(r_*) \to E^{cs}_*(r_*)$ is a $C^{1 + \Lip}$-mapping with $\calLip(\gamma_x^n) \leq 1$.
\item[(b)] The partition $\Xi^n$ of $\Sc^n = \cup_{x \in \Lambda^n} \xi^n(x)$ into leaves $\xi^n(x) := \exp_{x_*} \graph \gamma^n_x$ is measurable. 
\end{itemize}
\end{lem}

\begin{proof}[Proof of Lemma \ref{lem:construct112stack}]
For (a) we apply Lemma \ref{lem:chartSwitch11} to switch the axes of $W^n_{(x, \uo)} = \exp_x \graph \check h_x$ to the common axes $E^u_*(r_*), E^{cs}_*(r_*)$, having already verified conditions (i)--(iii) in Lemma \ref{lem:chartSwitch11} 
in paragraphs (B) and (C).

For (b), measurability of $\Xi^n$ follows from the fact that $\Sc^n$ is the union of at-most finitely many sets of the form
\[
\mathscr S^n_x := \bigg( \bigcup_{z \in \Lambda^n \cap f^n_{\theta^{-n} \uo} \Bc(x_{-n}, \iota_n)} W^n_{(z, \uo)} \bigg)  \cap \exp_{x_*} \big( E_*(r_*) \big)  \, , \quad x \in \Lambda^n\, , 
\]
where $\iota_n=\iota_n(x)$ is chosen so that $z \mapsto \gamma_z^n$ varies continuously over
$z \in \Lambda^n \cap f^n_{\theta^{-n} \uo} \Bc(x_{-n}, \iota_n)$.
\end{proof}


\subsection{Limiting properties of $W^n$-leaves}\label{subsubsec:approxUleafReg}

Now that we have grouped nearby $W^u$ and truncated $W^n$ leaves into `stacks' $\Sc$ and $\Sc^n$,
we turn our attention to the limiting properties of $\Sc^n$ and its relation to $\Sc$. 
Recall that for $x \in \Lambda^n$ and $y \in \Lambda$, 
$\gamma^n_x, \gamma_y: E_*^u(r_*) \to E_*^{cs}( r_*)$ are the graphing functions 
of the leaves of $\Sc^n$ and $\Sc$ through $x$ and $y$ respectively.

\begin{prop}\label{lem:ctyInX}
For any $\e > 0$, there exist $\tilde n_0 = \tilde n_0(\e) \geq N_*$ and $\tilde \eta = \tilde \eta(\e) > 0$ with the following property. For any $n = n_i \geq \tilde n_0$ and any $x \in \Lambda^n, y \in \Lambda$ with $d(x,y) < \tilde \eta$, we have that 
$\| \gamma^n_x - \gamma_y\| < \e$ where $\| \cdot \|$ refers to the uniform norm on $C(E_*^u(r_*), E_*^{cs}( r_*))$.
\end{prop}

It follows that $\limsup_{n \to \infty} \Sc^n \subset \Sc$. The statement of Proposition \ref{lem:ctyInX}
is all that we need; we do not prove, nor use, 
the continuity of $x \mapsto \gamma^n_x$. However, it follows from a version of the arguments below that `oscillations' in the Hausdorff distance between nearby $\xi^n$-leaves can be made uniformly, arbitrarily small for all sufficiently large $n$. Indeed, the following is a modification of a standard argument for proving the continuity of actual $W^u$-leaves (see, e.g., Section 5 in \cite{blumenthal2017entropy}).

\begin{proof}[Proof of Proposition \ref{lem:ctyInX}] In this proof, we will assume as before a canonical identification of the 
tangent spaces at $x$ and $y$, which are very close. Also, we will, for simplicity, 
use the notation of two-sided charts, assuming $(x,\uo')$ and $(y, \uo)$ 
are such that
 $\omega_i = \omega'_i$ for all $i > -n$ where $n$ is as in the Proposition.
No relation between $\omega_i $ and $ \omega'_i$ for $i\leq-n$ is assumed, as that will depend on the Selection Lemma.

\smallskip

{\it Plan of proof.} Let ${\bf 0} : \R^u \to \R^{cs}$ denote the zero function.
We consider ${\bf 0}$ as a function in the chart at $\tau^{-k}(y, \uo)$ for some
$k \ll \tilde n_0$ {(both $k$ and $\tilde n_0$ to be determined)}, 
and let ${\bf 0}^k_y : B^u(\d l_0^{-1}) \to \R^{cs}$
be given by ${\bf 0}^k_y := \Tc_{\tau^{-1} (y, \uo)} \circ \cdots \circ \Tc_{\tau^{-k}(y, \uo)} {\bf 0}$
(where $\Tc_{\cdots}$ is the graph transform as in Section 2.3). 
Likewise, we consider ${\bf 0}$ as a function in the chart at 
$\tau^{-k}(x, \uo')$, and let ${\bf 0}^{k, n}_x : B^u(\d l_0^{-1}) \to \R^{cs}$ be given by 
${\bf 0}^{k, n}_x := \Tc_{\tau^{-1} (x, \uo')} \circ \cdots \circ \Tc_{\tau^{-k}(x, \uo')} {\bf 0}$.
{Leaving it to the reader to check that} the switching of axes (Lemma \ref{lem:chartSwitch11}) can be performed, we obtain two functions
$\tilde \gamma^{k}_y, \tilde \gamma^{n,k}_x: E_*^u ( r_*) \to E_*^{cs}( r_*)$ whose graphs
are contained in $\exp_y(\graph {\bf 0}^k_y)$ and $\exp_x(\graph {\bf 0}^{k,n}_x)$ respectively. We will bound $\| \gamma^n_x - \gamma_y \|$ via the triangle inequality
\begin{equation} \label{triangle}
\| \gamma^n_x - \gamma_y \| \leq \| \gamma^n_x - \tilde \gamma^{ n,k}_x \| + \| \tilde \gamma^{n,k}_x - \tilde \gamma^{k}_y \| + \| \tilde \gamma^{k}_y - \gamma_y \|  \ .
\end{equation}
For given $ \e>0$, to prove $\| \gamma^n_x - \gamma_y \|
< \e$ for all $n \ge \tilde n_0$, we plan to first choose $k=k(\e, l_0)$ and then $\tilde n_0=\tilde n_0(k,\e, l_0)$. 

{We isolate below another `change-of-chart' type estimate that will be used several times 
in the proof of (\ref{triangle}). The proof is straightforward and left to the reader.

\begin{lem}\label{lem:chartSwitchRedux} 
Let $g_1, g_2 : B^u(\d l(y, \uo)^{-1}) \to \R^{cs}$ be Lipschitz graphing maps in the chart at
$(y, \uo)$. For $i=1,2$, let
$\check g_i := (L_{(y, \uo)} \circ g_i \circ L_{(y, \uo)}^{-1})|_{E^u_{(y, \uo)}(2 r_*)}$, and
assume that $\graph \check g_i \subset E^u_{(y, \uo)}(2 r_*) \times E^{cs}_{(y, \uo)}(\frac12 r_*)$
with $\calLip(\check g_i) \leq 1/10$. Let $\gamma_i : E^u_*(r_*) \to E^{cs}_*$ be such that 
$\exp_{x_*} \graph \gamma_i$ $ = \exp_{x_*} E_*(r_*) \cap \exp_{y} \graph \check g_i$. 
Then $\calLip(\gamma_i) \leq 1/5$, and 
\begin{align}\label{eq:estNorm}
\| \gamma_1 - \gamma_2 \| \leq \tilde C | g_1 - g_2| \, ,
\end{align}
where $\tilde C = \tilde C(l_0) > 0$. The same holds when $(y, \uo)$ is replaced by $(x, \uo')$.
\end{lem}
}

{\it First and third terms in \eqref{triangle}: } We use the contraction estimate in Lemma \ref{lem:graphTransform} to obtain
$$
| {\bf 0}^{{k}}_y- g_{(y, \uo)} | \leq c^{k}
$$ 
where $c$ is as in Lemma \ref{lem:graphTransform} and $g_{(y, \uo)}$ is the graphing map
of the unstable manifold {in the chart at $(y, \uo)$}. By Lemma \ref{lem:chartSwitchRedux}, 
$$
\| \tilde \gamma^{k}_y - \gamma_y \| \leq \tilde C | {\bf 0}^{{k}}_y- g_{(y, \uo)} | \leq \tilde C c^k \, .
$$
We require $k$ to be large enough that $\tilde C c^{k}< \e/3$.

The first term on the right side of (\ref{triangle}), $\| \gamma^n_x - \tilde \gamma^{ n,k}_x \|$,
 is treated similarly, provided that 
$n-k$ is large enough that in the chart at $\tau^{-k}(x,\uo')$, 
$\Lip(\Tc_{\tau^{-(k+1)} (x,\uo')} \circ \cdots \circ \Tc_{\tau^{-n}(x,\uo')} h^-_x) \le \frac{1}{10}$.
This requires that $\tilde n_0 \ge k+ m_0+ m_1$ where $m_0, m_1$ are as in Proposition \ref{lem:inclination} and depend on $r_-$ and $K_-$.

Let $k=k(\e)$ be fixed from here on.
\medskip

{
{\it Second term in \eqref{triangle}.} Given $0 < \bar \e \ll 1$ to be determined,
we claim that for $\tilde \eta$ small enough and $\tilde n_0$ large enough
depending on $l_0, k$ and $\bar \e$, the following hold for 
$x,y$ with $d(x,y) < \tilde \eta$:
\begin{itemize}
\item[(a)] $d(f^{-k}_\uo x, f^{-k}_\uo y), \ 
d_H(E^{u/cs}_{\tau^{-k}(x, \uo')}, E^{u/cs}_{\tau^{-k}(y, \uo)} ) < \bar \e$, and
\item[(b)] 
$f^{-i}_{\uo'} x \in \Phi_{\tau^{-i} (y, \uo)} B \big(\frac12 r_*  l ( \tau^{-i}(y, \uo))^{-1} \big) 
\quad \mbox{for all } 0 \leq i \leq k\ .$
\end{itemize}
Item (b) and $d(f^{-k}_\uo x, f^{-k}_\uo y) < \bar \e$ follow from the fact that
$\calLip(f^{-i}_\uo) \leq \| df^{-i}_{\uo}\| \leq l_0^k e^{\frac{k(k+1)}{2} \d_2}$ and
$l(\tau^{-i}(y, \uo)) \leq l_0 e^{k \d_2}$, and that both bounds depend on $l_0$ 
and $k$
alone. To control $d_H(E^{u/cs}_{\tau^{-k}(x, \uo')}, E^{u/cs}_{\tau^{-k}(y, \uo)} )$, 
we apply Proposition \ref{prop:ctyEucs} to $\tau^{-k} (y, \uo), \tau^{-k}(x, \uo')
\in \{ l \leq l_0 e^{k \d_2}\}$, and require that $\tilde n_0 \ge n_0+k$, where 
$n_0 = n_0(\bar \epsilon, l_0 e^{k \d_2})$ is as in Proposition \ref{prop:ctyEucs}.

Now let ${\bf 0}^y_x : B^u(\d l( \tau^{-k}(y, \uo))^{-1}) \to \R^{cs}$ be the function
whose graph is the component of 
$(\Phi_{\tau^{-k}(y, \uo)})^{-1} \exp_{f^{-k}_\uo x} E^u_{\tau^{-k}(x, \uo')}$ 
in $B(\d l( \tau^{-k}(y, \uo))^{-1})$
containing $(\Phi_{\tau^{-k}(y, \uo)})^{-1} f^{-k}_\uo x$. By choosing $\bar \e$
sufficiently small, we may assume, by item (a) above, that $\Lip({\bf 0}^y_x) \leq 1/10$, 
$ \Lip(d {\bf 0}^y_x) \leq 1$, and $|{\bf 0} - {\bf 0}^y_x|$ is as small as we wish.
This together with item (b) permits us to apply Lemma \ref{lem:graphTransform}(b)
to ensure that the graph transform 
$$
{\bf 0}^{y, k, n}_x := \Tc_{\tau^{-1} (y, \uo)} \circ \cdots \circ \Tc_{\tau^{-k} (y, \uo)} 
{\bf 0}^y_x : B^u(\d l(y, \uo)^{-1}) \to \R^{cs}
$$
is well defined. Moreover, with the modulus of continuity of $\Tc_{\tau^{-i} (y, \uo)}, 
1\le i \le k$, depending only on $l_0 e^{k\d_2}$, we may choose $\bar \e$ sufficiently
small to guarantee that $|{\bf 0}^{y, k, n}_x - {\bf 0}^k_y| \le \frac{\e}{3\tilde C}$
where $\e$ is as in \eqref{triangle} and $\tilde C$ is as in 
Lemma \ref{lem:chartSwitchRedux}.

The $(\exp_{x_*}^{-1}\circ \Phi_{(y, \uo)})$-images of the graphs of 
${\bf 0}^{y, k, n}_x$ and ${\bf 0}^k_y$, when restricted
to $E_*(r_*)$ are precisely the graphs of $\tilde \gamma^{n,k}_x$ and
$\tilde \gamma^{k}_y$ respectively. Another application of Lemma \ref{lem:chartSwitchRedux} gives
 $\| \tilde \gamma^{n,k}_x - \tilde \gamma^{k}_y \| \le \frac{\e}{3}$.}
\end{proof}

{\it For each $\uo \in \Gc$, the constructions of Section 5 are fixed for the remainder of the paper.}


\section{Proof of SRB property}

We now complete the proof of the Main Proposition (Proposition \ref{prop:main}).

\subsection{Construction of partitions respecting unstable manifolds}

Let $\Sc$ be as in Sect. 5.1, paragraph (D), i.e.,
$\Sc$ is a stack of local unstable manifolds through points in 
$\Lambda$, with $\Xi$ denoting the partition
into unstable leaves. To capture the conditional measures on $\Xi$ of any measure 
$\nu$ supported on $\Sc$, a standard procedure is to construct a sequence
of finite partitions $\a_1, \a_2, \cdots$ of $\Sc$ with the following properties:
\begin{itemize}
\item[(a)] The sequence $\{ \a_m \}$ is increasing, i.e. $\a_{m+1} \geq \a_m$ for all $m \geq 1$.
\item[(b)] For each $m$, we have $\a_m \leq \Xi$, i.e., $\a_m$ consists of intact $\xi$-leaves; and
\item[(c)] $\vee_{m= 1}^\infty \a_m \circeq \Xi$, where $\circeq$ denotes equivalence mod $0$ with respect to $\nu$.
\end{itemize}
Then properties of the conditional measures of $\nu$ on $\Xi$ can be deduced
from its conditional measures on $\alpha_m$ as $m \to \infty$.

\medskip
Complicating matters in our setting is that the measure of interest is the limit of
a sequence of measures that are not supported on $\Sc$ but on nearby
stacks $\Sc^n$ of $W^n$ leaves; see Section 5.
To accommodate these approximating measures, we will construct partitions 
similar to $\alpha_m$ but with slightly ``enlarged" elements, so they will contain intact
$W^n$ leaves. 
The aim of this subsection is to make precise the construction of such a sequence of partitions we will call $\beta_m$.

Continuing to use notation from Section 5, we define 
$\tilde \mu^n_\uo = \mu^n_\uo|_{\Lambda^n}$. On refining the sequence $\{ n_i \}$, let us assume that $\tilde \mu^n_\uo$ converges weakly as $n \to \infty$ to a measure $\tilde \mu_\uo$. Note that $\tilde \mu_\uo \leq \mu_\uo$, and that $\tilde \mu_\uo$ is supported on $\Lambda$ (by Lemma \ref{lem:convergeUnifSet}), with $\tilde \mu_\uo(\Lambda) \geq   c_* > 0$ (Lemma \ref{lem:accumulateMass}). For $S \subset \exp_{x_*}( E_*(r_*))$ 
let us write
\[
\diam^{cs}(S) = \sup_{u \in E_*^u(r_*)} \diam( S \cap \exp_{x_*}(u + E_*^{cs}(r_*))) \, .
\]

{
\begin{lem}\label{lem:UsetConstruct}
There is a decreasing sequence of compact subsets $\Delta_1 \supset \Delta_2
\supset \cdots$ of $\Lambda$ with the following properties: 
\begin{itemize}
\item[(i)] Each $\Delta_m$ is partitioned into disjoint compact subsets
$\{\Delta_{m,k}, 1 \le k \le K_m\}$ and the $\{\Delta_{m,k}\}$ are nested in
the sense that each $\Delta_{m+1, k} \subset \Delta_{m, k'}$ for some $k'$.
\item[(ii)] The sets $\Sc_{m, k} := \cup_{x \in \Delta_{m, k}} \xi(x)$ are compact and pairwise disjoint among $1 \leq k \leq K_m$.
\item[(iii)] We have \[\lim_{m \to \infty} \max_{1 \leq k \leq K_m} \diam^{cs}(\Sc_{m, k}) = 0 \, . \]
\item[(iv)] Defining $\Delta_\infty = \cap_{m \geq 1} \Delta_m$, we have
$
\tilde \mu_\uo(\Delta_\infty) \geq \frac12 c_* \, .
$
\end{itemize}
\end{lem}
}

\smallskip
\begin{proof}
For ease of notation, in the following proof, let us suppress the ``$\uo$'' and write $\tilde \mu := \tilde \mu_\uo$.

\medskip

Define $\Sigma = \exp_{x_*} E_*^{cs}(r_*)$, which as is easily checked is a transversal to the $\Xi$-leaves comprising $\Sc$. Set $\hat \Sigma = \Sigma \cap \Sc$ and let $\pi : \Sc \to \hat \Sigma$ denote the projection along $\Xi$-leaves. Project $\tilde \mu$ to its transverse measure $\tilde \mu^T = \tilde \mu^T_\uo$ on $\hat \Sigma$.
For each $m \geq 1$, let $\mathcal Q_m$ be a partition of $\Sigma$ into cubes of side lengths $\approx 1/2^m$ with the following properties:
\begin{itemize}
\item[(1)] The sequence $\mathcal Q_m, m \geq 1$ is \emph{increasing}, i.e., $\Qc_{m+1} \geq \Qc_m$ for each $m \geq 1$.
\item[(2)] We have $\vee_{m = 1}^\infty \Qc_m \circeq \varepsilon$, the partition of $\hat \Sigma$ into points $\tilde \mu^T$-mod 0; and
\item[(3)] For each $C \in \Qc_m$, we have $\tilde \mu^T(\pd C) = 0$.
\end{itemize}
With the $\Qc_m$ fixed, we define finite collections $\check \Qc_m, m \geq 1$, of disjoint compact sets via the following inductive procedure. 
Fix an increasing sequence $c_1 \leq c_2 \leq \cdots < 1$ for which 
$\prod_{m=1}^\infty c_m = \frac12$. 
To start, for each $C \in \Qc_1$ fix a compact subset $\check C \subset C \cap \pi(\Lambda)$ for which $\dist(\check C, \pd C) > 0$ and $\tilde \mu(\check C) \geq c_1 \tilde \mu(C)$. We set $\check \Qc_1 = \{ \check C : C \in \Qc_1\}$, so that
\[
\tilde \mu^T \bigcup_{C \in \Qc_1} \check C \geq c_1 \, \tilde \mu^T(\hat \Sigma) \geq c_1 \,  c_* \, .
\]
We construct successively $\check \Qc_1, \check \Qc_2, \cdots$ of disjoint compact subsets 
with the rule that
\begin{itemize}
\item[(i)] for each $\check C \in \check \Qc_{m}$, we have that $\check C \subset \check C'$ for some $\check C' \in \Qc_{m - 1}$, and
\item[(ii)] for each $\check C' \in \check \Qc_{m-1}$, we have $\tilde \mu^T( \cup_{\check C \in \check \Qc_{m}, \check C \subset \check C'} \check C) \geq c_{m} \tilde \mu^T(\check C')$.
\end{itemize}
With the $\{\check \Qc_m\}$ in hand, we now define the array of compact sets $\Delta_{m, k}$ as follows: for each $m \geq 1$ and $1 \leq k \leq K_m := |\Qc_m|$, we define $\Delta_{m, 1}, \cdots, \Delta_{m, K_m}$ to be the collection of sets of the form $\Delta \cap \pi^{-1}(\check C)$ as $C$ ranges over 
$\check \Qc_m$.

Item (i)--(iv) follow immediately. 
\end{proof}

What we have done in Lemma \ref{lem:UsetConstruct} is to group the unstable leaves in $\Sc$ into
finer and finer substacks with a Cantor-like structure transversally, and to do that,
we have had to give up on a little bit of $\tilde \mu_\uo$-measure. Let 
$$
\Sc_\infty := \bigcup_{x \in \Delta_\infty} \xi(x)\ .
$$

\begin{cor}\label{cor:separatingOpens} There is a decreasing sequence of  open sets 
$$
U_1 \supset U_2 \supset \dots \qquad \mbox{with} \qquad
\cap_m U_m = \Sc_\infty
$$
and a sequence of partitions $\beta_m = \{\beta_{m,k}\}$ of $U_m$ into finitely
many disjoint open sets, 
with the properties that
\begin{itemize}
\item[(i)] the partitions $\beta_m$ are nested, i.e., each $\beta_{m,k}
\subset \beta_{m-1, k'}$ for some $k'$;
\item[(ii)] each $\beta_{m,k}$ contains intact leaves of the compact substack 
$\Sc_{m,k}$, and
\item[(iii)] for each $\xi$ in $\Sc_\infty$, if $\beta_{m,k}(\xi)$ is the element 
of $\beta_m$ containing $\xi$, then $\beta_{m,k}(\xi) \downarrow \xi$ as $m \to \infty$.
\end{itemize}
\end{cor}

\smallskip
Corollary \ref{cor:separatingOpens} follows easily from Lemma \ref{lem:UsetConstruct}. The sets $\{\beta_{m,k}\}$ can be
chosen quite arbitrarily as long as they have the stated properties.

\subsection{Pushed forward measures and their conditional densities}

Recall that for each $n=n_i$, we have constructed a stack $\Sc^n$ and
a partition of $\Sc^n$ into sets $\xi^n(\cdot)$ that are approximate $W^u$-leaves 
(Lemma \ref{lem:construct112stack}).
The next lemma establishes that for each $m$, by taking $n$ large enough, 
the partition $\beta_m$ will respect a definite fraction of $\xi^n$-leaves. 
Let $\Nc_\eta(\cdot)$ denote the $\eta$-neighborhood of a set. 

\begin{lem}\label{lem:intactLeavesInsideBeta}
For each $m \geq 1$, there exist $\eta_m > 0$ and $N_m \in \N$ with the property that the following hold for all $n=n_i \ge N_m$.
\begin{itemize}
\item[(a)] Define
\[
\Lambda^n_{m, k} := \Lambda^n \cap \overline{\Nc_{\eta_m}(\Delta_{m, k})} \, , \quad \Sc^n_{m, k} := \bigcup_{x \in \Lambda^n_{m, k} } \xi^n(x) \, .
\]
Then $\Sc^n_{m, k} \subset \beta_{m, k}$.
\item[(b)] Letting $\Lambda^n_m = \cup_k \Lambda^n_{m,k}$, we have
$$
\tilde \mu^n_\uo(\Lambda^n_{m, k}) \geq  \frac23  \tilde \mu_\uo (\Lambda_{m, k})\ ,
\qquad \mbox{hence} \qquad  \tilde \mu^n_\uo (\Lambda^n_m) \ge \frac13  c_* \ .
$$
\end{itemize}
\end{lem}

\begin{proof} (a) follows immediately from Proposition \ref{lem:ctyInX} and (b) from
the fact that $\tilde \mu^n_\uo$ is assumed to converge to $\tilde \mu_\uo$.
\end{proof}

While we have used $\mu^n_\uo|_{\Lambda^n_m}$ to ensure that our partitions 
are catching a definite fraction of $\mu^n_\uo$, we are primarily interested in
$\mu^n_\uo|_{\Sc^n_m}$ where $\Sc^n_m=\cup_k \Sc^n_{m, k}$. We now turn
our attention to $\mu^n_\uo|_{\Sc^n_m}$, focusing on a part of this measure 
with controlled conditional densities.

Recall from Sect. 4.3 that we disintegrate $\mu$ on the leaves of a foliation 
$\Fc_-^n$ 
to be carried forward by $f^n_{\theta^{-n} \uo}$, and that $\Fc_-^n$ is
defined on a ball $B \subset \{ \psi \geq \a_0 / 2\}$, where $\psi = \frac{d \mu}{d \Leb}$.
Define
$$
\nu_-^n := \frac{\a_0}{2} \Leb|_B, \qquad \mbox{and} \qquad 
\nu^n = \big( f^n_{\theta^{-n} \uo} \, \,  \nu^n_- \big)  \big|_{\Sc^n}\ ,
$$
Since $\nu_-^n \le \mu$, we have that $\nu^n \leq \mu^n_\uo$ for all $n$.

Furthermore, let $(\nu^n_{\xi})_{\xi \in \Xi^n}$ denote the (normalized)
disintegration measures of $\nu^n$ along $\Xi^n$ with transversal measure 
$\nu^n_T$ on $\Sc^n / \Xi^n$. For 
$\xi \in \Xi^n$, let $\phi(\xi)$ denote the leaf of $\Fc_-^n$ containing $f^{-n}_\uo \xi$. 
It is easy to see that $\nu^n_\xi$ is $(f^n_{\theta^{-n} \uo}  \Leb_{\phi(\xi)}) |_{\xi}$ normalized. 

\begin{lem}\label{prop:defineMeasureNu} \
\begin{itemize}
\item[(a)] For measurable $C \subset \Lambda^n$, 
$\nu^n(C) \geq \frac{\a_0}{2\| \psi \|_\infty} \cdot \mu^n_\uo(C)$.
\item[(b)] For a.e. $\xi \in \Xi^n$, $\nu_\xi^n$ is absolutely continuous with 
density $\rho_\xi^n : \xi \to (0, \infty)$ satisfying the distortion estimate
\[
\bigg| \log \frac{\rho_\xi^n(p_1)}{\rho_\xi^n(p_2)} \bigg| \leq \bar D 
\]
for any $p_1, p_2 \in \xi$. Here $\bar D = \bar D(l_0, K_-, r_-) > 0$, where $K_-, r_-$ 
are as in Lemma \ref{lem:lamination}; 
in particular, $\bar D$ is independent of $\xi$ and $n$.
\end{itemize}
\end{lem}

\begin{proof}
The estimate in (a) follows from the simple bound $\mu^n_\uo (C) \leq \| \psi\|_{\infty} \Leb(f^{-n}_\uo C)$, 
and Item (b) 
follows from the distortion estimate in Lemma \ref{lem:distortion} applied to 
$K_0 = K_-, r_0 = r_-$.
\end{proof}

\subsection{Passing to the weak limit as $n \to \infty$ and completing the proof}

Let $\nu^n_m := \nu^n|_{\Sc^n_m}$. From Lemmas \ref{lem:intactLeavesInsideBeta}(b) and \ref{prop:defineMeasureNu}(a), we have
that for every $m$,  
\begin{align}\label{lower1112}
\nu^n_m(\Sc^n_m) \ge \frac{\a_0}{6\| \psi \|_\infty} 
 c_*
\end{align}
for every $n=n_i \ge N_m$. We fix such an $n(m)$ for each $m$; clearly
$n(m) \to\infty$ as $m \to \infty$. Let
$\nu^*$ be any limit point of the sequence $\nu^{n(m)}_m$ as $m \to \infty$.
Then $\nu^*$ is supported on $\Sc_\infty$ with $\nu^* \leq \mu_\uo$. Moreover,
the lower bound \eqref{lower1112} passes to $\nu^*(\Sc_\infty)$.
Let $\nu^*_\xi$ denote
the conditional measures of $\nu^*$ on the leaves $\xi \in \Xi$. To complete
the proof of Proposition \ref{prop:main}, it suffices to show that for a.e. $\xi$, the measure $\nu^*_\xi$
is absolutely continuous. 


For this, we state below a lemma that will be used to deduce properties of 
the conditional measures 
of $\nu^*$ on leaves of $\Xi$ from those of $\nu^n_m$ on $\Xi^n$. First, we need some
 notation: Let $C^u \subset E^u_*( r_*)$ be a cube. We let
$\hat{\Leb}(C^u) = \Leb(C^u)/\Leb(E_*^u( r_* ))$, and
define
$$
V_{C^u} := \exp_{x_*}(C^u + E^{cs}_*( r_*)) \ .
$$

Recall that $\nu^n_m|_{\Sc^n_{m,k}} =\nu^n_m|_{\beta_{m,k}}$ for $n=n(m)$. 

\begin{lem}\label{lem:proveDensityBounds} There exists $A>1$ such that for
any $C^u \subset E_*^u( r_*)$, we have, for all large enough 
$m$ and $n=n(m)$:
\begin{equation} \label{density bounds}
\frac{1}{A} \cdot \hat{\Leb}(C^u)  \ \le \ 
\frac{\nu_m^n(\beta_{m,k} \cap V_{C^u})}{\nu_m^n(\beta_{m,k})} \ \le \ 
A \cdot \hat{\Leb}(C^u) \ .
\end{equation}
It follows that for a.e. $\xi \in \Xi$, 
{we have $\pi^u_*\nu^*_\xi \ll \hat \Leb$ with 
$\frac{d (\pi^u_*\nu_\xi^*)}{d \hat \Leb} \in  [\frac{1}{ A},  A]$, where
$\pi^u$ is projection onto $E^u_*(r_*)$ along $E^{cs}_*$.}
\end{lem}


\begin{proof}[Proof of Lemma \ref{lem:proveDensityBounds}] 
Let $C^u$ be fixed. That (\ref{density bounds}) holds for each $\nu_m^n$ follows
from the fact that it holds for each $\nu^n_\xi$ 
by Lemma \ref{prop:defineMeasureNu}(b).

To pass these bounds to $\nu^*_\xi$,
we let $m$ and $k$ be fixed to begin with. Since 
$$
\nu^{n(m')}_{m'}(\beta_{m,k} \cap V_{C^u}) = \sum_{k' : \beta_{m',k'} \subset \beta_{m,k}}
\nu^{n(m')}_{m'}(\beta_{m',k'} \cap V_{C^u})\ ,
$$
and (\ref{density bounds}) holds for all $m' \ge m$ and all $k'$, 
it follows that for fixed $\beta_{m,k}$, (\ref{density bounds}) holds with $\nu^n_m$ replaced by $\nu^{n(m')}_{m'}$. Letting $m' \to \infty$, 
we obtain that it holds with $\nu^*$ in the place of $\nu^n_m$.

Continuing to keep $C^u$ fixed but letting $m \to \infty$ and running through
all $\beta_{m,k} \in \beta_m$ for each $m$, we obtain by Corollary \ref{cor:separatingOpens}(iii)
that for a.e. $\xi \in \Xi$, 
$$
\frac{1}{A} \cdot \hat{\Leb}(C^u) \ \le {\pi^u_*\nu^*_\xi (C^u)}\ \le
A \cdot \hat{\Leb}(C^u)\ .
$$
As cubes form a basis for the topology on $E_*^u(r_*)$,
the assertion follows. 
\end{proof}

The proof of Proposition \ref{prop:main} is now complete.

\section*{Appendix}

Below, we deduce Lemma \ref{lem:measSelection} from the following well-known theorem.

\begin{thm}[Kuratowski-Ryll-Nardzewski measurable selection theorem]\label{thm:measurableSelection}
Let $(\Xi, \mathcal M)$ be a measurable space, $Z$ a Polish space, and let $F : \Xi \to 2^Z$ be a set-valued mapping such that
\begin{itemize}
\item $F(\xi)$ is closed and nonempty for each $\xi \in \Xi$, and
\item for any open $U \subset Z$, we have
\[
\{ \xi \in \Xi : F(\xi) \cap U \neq \emptyset\} \in \mathcal{M} \, .
\]
\end{itemize}
Then, there exists a measurable map $f : \Xi  \to Z$ for which $f(\xi) \in F(\xi)$ for all $\xi \in \Xi$.
\end{thm}

For an account of measurable selection theorems, see, e.g., the survey \cite{wagner1977survey}.

\begin{proof}[Proof of Lemma \ref{lem:measSelection}] Let $X, Y$ be Polish and let $G \subset X \times Y$ be a compact subset, writing $G_X$ for the projection of $G$ onto $X$. Applying Theorem \ref{thm:measurableSelection} to $(\Xi, \mathcal M) = (X, \operatorname{Bor}(X))$, $Z = Y$ and
$F(x) := \{ y \in Y : (x,y) \in G\}$, it suffices to show that for any open $U \subset Y$,
\[
V_U := \{ x \in X : F(x) \cap U \neq \emptyset \}
\]
is a Borel measurable subset of $X$.

For this, note that because $Y$ is Polish, we may represent $U$ as the countable union of closed sets $U_i$, so $U = \cup_i U_i$. Moreover, as one easily checks,
\[
V_U = \bigcup_i V_{U_i} \, .
\]
It suffices to show that each $V_{U_i}$ is closed. For this, let $\{ x^n\} \subset V_{U_i}$ be a sequence converging to a point $x \in X$. To show $x \in V_{U_i}$, fix for each $n$ an element $y^n \in F(x^n) \cap U_i$. By compactness of $G \cap (X \times U_i)$, it follows that a subsequence of $(x^n, y^n)$ converges to an element $(x^*, y^*)$ of $G \cap (X \times U_i)$. But $x = x^*$, hence $y^* \in F(x)$; since $y^* \in U_i$, it follows that $x \in V_{U_i}$.
\end{proof}

\bibliography{biblio}
\bibliographystyle{plain}

\end{document}